\documentclass[11pt, reqno]{amsart}
  \usepackage{fullpage}
  \usepackage{thmtools}
\usepackage{thm-restate}

 \newif\ifHideFoot
\HideFoottrue   \HideFootfalse

\usepackage{enumitem, amssymb,amsmath,amsthm,amscd,mathrsfs,graphicx, color}
\usepackage[cmtip,all,matrix,arrow,tips,curve]{xy}
  \usepackage{xr-hyper}
\usepackage{hyperref}
\usepackage{cleveref}

\makeatletter\@input{fakediagonal.tex}\makeatother
\makeatletter\@input{fakebloch.tex}\makeatother

\usepackage[usenames,dvipsnames]{xcolor}
\usepackage{xypic}
\hypersetup{colorlinks=true,citecolor=ForestGreen,linkcolor=Maroon,urlcolor=NavyBlue}
 \usepackage{mathtools}
\usepackage{mathpazo}

\setlist[enumerate]{leftmargin=8.5mm}

\numberwithin{equation}{section}
\newtheorem{teo}{Theorem}[section]
\newtheorem{pro}[teo]{Proposition}
\newtheorem{lem}[teo]{Lemma}
\newtheorem{cor}[teo]{Corollary}

\newtheorem*{teo*}{Theorem}
\newtheorem*{pro*}{Proposition}

\newtheorem{teoalpha}{Theorem}

\theoremstyle{definition}

\newtheorem{exa}[teo]{Example}

\theoremstyle{remark}
\newtheorem{rem}[teo]{Remark}

\ifHideFoot

\newcommand{\Yano}[1]{}
\newcommand{\Jeff}[1]{}
\newcommand{\Charles}[1]{}

\else

\newcommand{\marg}[1]{\normalsize{{
   \color{red}\footnote{{\color{blue}#1}}}{\marginpar[\vskip
   -.25cm{\color{Maroon}\hfill\thefootnote$\implies$}]{\vskip
    -.2cm{\color{Maroon}$\impliedby$\tiny\thefootnote}}}}}
\newcommand{\Yano}[1]{\marg{(Yano) #1}}
\newcommand{\Jeff}[1]{\marg{(Jeff) #1}}
\newcommand{\Charles}[1]{\marg{(Charles) #1}}

\fi

  \def\res{{\bf\rm R}}
\def\sep{{\text{sep}}}

\DeclareMathOperator{\im}{im}

\global\let\hom\undefined
\DeclareMathOperator{\hom}{Hom}

\newcommand{\til}[1]{{\widetilde{#1}}}

\def\ff{{\mathbb F}}

\def\rat{{\mathbb Q}}
\def\integ{{\mathbb Z}}

\def\calo{{\mathcal O}}

\def\iso{\cong}

\renewcommand{\bar}[1]{{\overline{#1}}}
\newcommand{\ubar}[1]{{\underline{#1}}}

\DeclareMathOperator{\aut}{Aut}
\DeclareMathOperator{\alb}{Alb}

\DeclareMathOperator{\End}{End}

\DeclareMathOperator{\gal}{Gal}

\DeclareMathOperator{\spec}{Spec}

\newcommand{\st}[1]{\left\{#1\right\}}
\newcommand{\abs}[1]{{\left|#1\right|}}

\newcommand{\floor}[1]{\lfloor #1 \rfloor}

\title{A complete answer to Albanese base change for  incomplete varieties}

\author{Jeffrey D. Achter}
\address{Colorado State University, Department of Mathematics,
 Fort Collins, CO 80523,
 USA}
\email{j.achter@colostate.edu}

\author{Sebastian Casalaina-Martin }
\address{University of Colorado, Department of Mathematics,
 Boulder, CO 80309, USA }
\email{casa@math.colorado.edu}

\author{Charles Vial}
\address{Universit\"at Bielefeld, Fakult\"at f\"ur Mathematik, Germany}
\email{vial@math.uni-bielefeld.de}

\thanks{Research of the first and second authors is supported in part by grants from the Simons Foundation (637075 and 581058, respectively).
The research of the third author was funded by the Deutsche Forschungsgemeinschaft (DFG, German Research Foundation) -- SFB-TRR 358/1 2023 -- 491392403}

\begin{document}

\begin{abstract}
Albanese varieties provide a standard tool in algebraic geometry for converting questions about varieties in general, to questions about Abelian varieties.  A result of Serre provides the existence of an Albanese variety for any  geometrically connected and geometrically reduced scheme of finite type over a field, and a result of Grothendieck--Conrad establishes that Albanese varieties are stable under base change of field provided the scheme is, in addition, proper.   A result of Raynaud shows that base change can fail for Albanese varieties without this properness hypothesis.   In this paper we show that Albanese varieties of  geometrically connected and geometrically reduced schemes of finite type over a field are stable under separable field extensions.  We also show that the failure of base change in general is explained by the L/K-image for purely inseparable extensions L/K.  
\end{abstract}


  \maketitle

  \section*{Introduction}

Consider a $K$-pointed 
         scheme  $(V,v)$ of finite type over a field $K$.    A \emph{pointed
Albanese variety} for this object, if it exists, consists of an abelian variety
$\alb_{V/K}$ over $K$ and a pointed $K$-morphism $$\xymatrix{a_v: V \ar[r]& \alb_{V/K}}$$ taking $v$ to
$\mathbf{0}_{\operatorname{Alb}_{V/K}}$, which is initial for pointed $K$-morphisms from $V$ to abelian
varieties.  More generally, if $V$ is a
 scheme of finite type over a field $K$,
an \emph{Albanese torsor} for $V$, if it exists, is a morphism $a:V \to \alb^1_{V/K}$ to a torsor under an abelian variety $\alb_{V/K}$, the \emph{Albanese variety} of~$V$, which is initial for morphisms to torsors under abelian varieties.
For complex projective manifolds, Albanese varieties were 
a classical, and central, tool in algebraic geometry; they provide a
method of converting geometric questions about a variety into
related questions about abelian varieties, and extend the techniques
used for studying  smooth projective curves via the Jacobian and the
Abel map to varieties of higher dimension.    In 1960, Serre showed
such an Albanese variety  exists for any geometrically
connected and geometrically reduced scheme of finite type over any
field    \cite{serrealb} (see \Cref{T:Serre-Alb} and \Cref{R:GrCon}),
thereby allowing for the extension of these classical techniques to
this setting.  Other   treatments were considered at about the same
time  \cite{NS52,Chevalley60}.

After existence, perhaps the most important structural question is  base change.
In the special case where $V$ is in addition assumed to be  proper and geometrically normal  over $K$, Grothendieck  \cite{FGA} 
identified $\alb_{V/K}$ with
$((\operatorname{Pic}^0_{V/K})_{\operatorname{red}})^\vee$, the dual abelian variety of the reduction of $\operatorname{Pic}^0_{V/K}$.  In this
setting, since the formation of the Picard scheme is compatible with
base change, it follows that if $L/K$ is any field extension, then the
canonical map $\beta_{V,L/K}:\alb_{V_L/L} \to (\alb_{V/K})_L$ is an
isomorphism. 
However, without the hypothesis that $V$ be proper, it is known that base change can fail:

\begin{exa}[Raynaud: Albaneses are not stable under base change]
  \label{E:badbc}
  Let $L/K$ be a finite purely inseparable field extension, and let  $A/L$ be an 
  abelian variety.   
      Let $G =
  \res_{L/K}A$ be the Weil restriction, which is a smooth  connected commutative algebraic
  $K$-group with  $\dim G= [L:K]\cdot \dim A$    \cite[Exp.~XVII, App.~II, Prop.~5.1]{SGA32}.  
 There is a short exact sequence of $L$-groups 
   \cite[Exp.~XVII, App.~II, Prop.~5.1]{SGA32}
 $$
 \xymatrix{1\ar[r]& U\ar[r]& G_L\ar[r]^u&A\ar[r]& 0,}
 $$  where $U$ is a smooth connected and unipotent linear algebraic $L$-group, 
 exhibiting $A$ as the Albanese of $G_L$  \cite[Ex.~4.2.7]{brionstructure}.  
   As we will see in \Cref{L:albBCinsep},   $\operatorname{Alb}_{G/K}=\im_{L/K} A$, the $L/K$-image of $A$ (see \S \ref{S:pfThmA}).  
 If $A$ is defined over $K$,   then $(\operatorname{Alb}_{G/K})_L=(\im_{L/K} A)_L=A=\operatorname{Alb}_{G_L/L}$.  However, if $A$ is \emph{not} defined over $K$, then $G$ is not an extension of an abelian variety over $K$ by a smooth algebraic $K$-group   
 \cite[Exp.~XVII, App.~II, Cor.(ii)~to~Prop.~5.1]{SGA32} \footnote{Note that the statement of  \cite[Exp.~XVII, App.~II, Cor.(ii)~to~Prop.~5.1]{SGA32} has the implicit hypothesis that $A$ not be defined over $K$: the proof uses this hypothesis to conclude that the linear algebraic subgroup  of $G$ in the proof is strictly larger than $U$ when base changed to $L$; moreover, the hypothesis that $A$ not be defined over $K$ is required in the statement of the corollary, as this example shows.}; Brion uses this to show that in this case $(\operatorname{Alb}_{G/K})_L   \ne A = \operatorname{Alb}_{G_L/L}$  \cite[Ex.~4.2.7]{brionstructure},
   providing an example where base change fails.     
\end{exa}

Despite this failure of base change, 
there are a few striking features of this example.  First, the field extension is purely inseparable, and second, it happens that   $\alb_{G/K} \iso \im_{L/K} \alb_{G_L/L}$.
  The main result of this paper shows that these observations about the Raynaud example represent the general situation.  In other words,   it is the
inseparability of $L/K$, and not the improperness of $G$, that drives
the failure of base change in the Raynaud example.  Indeed,  for a purely inseparable extension, the Albanese is the $L/K$-image of the Albanese of the base change:

\begin{teo*}\label{T:T}
   Let $V$ be a 
    geometrically connected and geometrically reduced 
  scheme of finite type over a field~$K$, 
  and let $L/K$ be an extension of fields.
  \begin{enumerate}[label={\bf(\Alph*)}]
  \item\label{T:main}  (Theorem \ref{T:mainAbody}) If $L/K$ is separable, then $\alb_{V_L/L} \iso (\alb_{V/K})_L$.
  \item\label{T:insep} (\Cref{L:albBCinsep}) If $L/K$ is a purely inseparable extension, then $\alb_{V/K} \iso \im_{L/K} \alb_{V_L/L}$.
  \end{enumerate}
\end{teo*}

Recall that any field extension $L/K$ factors as $L/L'/K$ with $L'/K$ separable and $L/L'$ purely inseparable (see, \emph{e.g.},  \S\ref{S:SepExt}), so that the theorem above completely describes base change for arbitrary field extensions.
We note that our proof of Theorem~\ref{T:main} relies on Theorem~\ref{T:insep}, due to our use of de Jong's regular alterations, and the fact that these regular alterations are only smooth over a purely inseparable extension of the base field.
Theorem \ref{T:main} generalizes the abelian (as opposed to semi-abelian) case
        of \cite[Cor.~A.5]{wittenberg08}, which requires $V$ to be an open
 subset of a smooth proper geometrically integral  scheme over
 $K$; it also generalizes \cite[Prop.~A.3(i)]{mochizuki12} to the case
 of non-perfect base fields, and non-algebraic field extensions.
 Theorem \ref{T:insep} completely explains the behavior studied in
 \cite[Ex.~4.2.7]{brionstructure}.   Theorem \ref{T:insep} implies
 the weaker statement that, for a purely inseparable extension $L/K$, $\alb_{V_L/L}$ and
 $(\alb_{V/K})_L$ differ by a purely inseparable isogeny.  This, as
 well as Theorem \ref{T:main}, has also been secured by Schr\"oer
\cite[Thm.~6.1]{schroeralbanese} for separated schemes, under a
hypothesis on the ring of global functions.

 We note that the geometrically connected hypothesis for $V$ in the
 theorem is necessary, as it is necessary for the existence of an
 Albanese (\Cref{C:GCX}).  In contrast, the geometric reducedness hypothesis in the theorem is more subtle (\Cref{E:NonRedAlb}, \Cref{E:NonRedAlbX}, \Cref{E:NonRedAlbBC}), although one can at least say that the reducedness of $\Gamma(V,\mathcal O_V)$ is necessary in Theorem~\ref{T:main}, and the geometric reducedness of $\Gamma(V,\mathcal O_V)$ is necessary in Theorem~\ref{T:insep}, as the reducedness (but \emph{not} geometric reducedness) of $\Gamma(V,\mathcal O_V)$ is necessary for the existence of an Albanese (\Cref{C:RX}, \Cref{E:GRX}). 
    See also  \Cref{C:Alb-nEx}\ref{C:Alb-1} where we summarize some necessary  conditions for the existence of an Albanese.

\medskip 
 In light of the modern treatment of Albaneses following Grothendieck, our definition of the Albanese, and our subsequent focus on  base change of field may seem slightly archaic.  Indeed,
 Grothendieck would require that the Albanese of $(V,v)$ satisfy the stronger condition that for any morphism of schemes $S\to K$,   and  any pointed $S$-morphism  $f: V_S \to A$ to an abelian
  scheme $A/S$ sending $v_S$ to $\mathbf 0_{A}$, there exists a unique  $S$-homomorphism $g: (\alb_{V/K})_S \to
  A$ such that $g\circ a_S = f$.  Such an Albanese would automatically satisfy arbitrary base change.   
 Grothendieck and Conrad show that for a  pointed \emph{proper}  geometrically connected and geometrically reduced scheme  $V$ of finite type over $K$, the Albanese as defined here satisfies this stronger condition.  Because of Theorem \ref{T:insep}, the
 best possible result
  along these lines 
  without the properness hypothesis 
 is:

  \setcounter{teoalpha}{2}
  \begin{teoalpha}(Theorem \ref{T:groconrad-open}) \label{T:C}
    Let $(V,v)$ be a $K$-pointed  (geometrically) connected and geometrically reduced 
   scheme of finite type over a field $K$.  
   Then for any (inverse limit of) smooth
   morphism of schemes $S  \to \spec K$, and any pointed $S$-morphism  $f: V_S \to A$ to an abelian
   scheme $A/S$ sending $v_S$ to $\mathbf 0_A$, there exists a unique  $S$-homomorphism $g: (\alb_{V/K})_S \to A$ such that $g\circ a_S = f$.
  \end{teoalpha}
We refer to Theorem \ref{T:groconrad-open} for the Albanese torsor version of Theorem~\ref{T:C} valid for (not necessarily $K$-pointed) geometrically connected and geometrically reduced 
scheme of finite type over a field~$K$.

 \medskip 
 
 We originally worked out these arguments as part of our development
 of a functorial approach to regular homomorphisms \cite{ACMVfunctor}.
 Indeed, Theorem \ref{T:main} originally appeared as an appendix to
 \emph{op.~cit.}.  However, since it seemed that these results on
 Albanese varieties might be useful to a wider audience, we decided to
 make them available in a separate document.  Since that preprint
 originally appeared, Laurent and Schr\"oer have studied the existence
 of a relative Albanese for proper families \cite{laurentschroeralbanese}. Moreover, in
 the context of schemes over a field Schr\"oer,
 using different techniques, has extended some of our results under
 the further hypothesis that the scheme be separated \cite{schroeralbanese}.  Combining our
 \Cref{C:Alb-nEx}\ref{C:Alb-1} with \cite[Thm.~p.2]{schroeralbanese}
 provides necessary and sufficient conditions for the existence of an
 Albanese for a \emph{separated} scheme of finite type over a field
 (\Cref{C:Alb-nEx}\ref{C:Alb-2}).  A formulation of Schr\"oer's base
 change result \cite[Thm.~p.4]{schroeralbanese} can be found in
 \Cref{P:Schroeer}.

\subsection*{Acknowledgements}
We are indebted to Brian Conrad for helpful
conversations and to David Grant for explaining the proof of 
\Cref{P:rkcondremix}\ref{P:rkcond2}.  We also thank Stefan Schr\"oer for a
detailed reading and useful comments.

 \section{Albanese varieties}\label{S:AlbaneseDef}

\subsection{Serre's existence theorem}
 Let $V$ be a scheme of
 finite type over a field $K$.
 Recall that an \emph{Albanese datum for $V$} consists of a triple 
 \begin{equation}\label{E:AlbDat}
 (\alb_{V/K}, \
 \alb^1_{V/K},\ a_{V/K}: V\to \operatorname{Alb}^1_{V/K} )
 \end{equation}
  with $\alb_{V/K}$ an abelian variety over
 $K$, $\alb^1_{V/K}$ a torsor under $\alb_{V/K}$ over $K$, and $a:V\to
 \alb^1_{V/K}$ a morphism of $K$-schemes which is initial, meaning  that
 given any triple $(A, P,f:V\to P)$ with $A$ an abelian variety over $K$, $P$ a
 torsor under $A$ over $K$, and $f:V\to P$ a morphism of $K$-schemes, there is a
 unique  $K$-morphism   $g:\operatorname{Alb}^1_{V/K}\to P$   making the following diagram
 commute:
 $$
 \xymatrix{
  V \ar[r]^<>(0.5){a_{V/K}} \ar[rd]_f&   \operatorname{Alb}^1_{V/K} \ar@{-->}[d]_{\exists !}^g\\
&P
 }
 $$
 We will respectively call the three objects in this datum the \emph{ Albanese variety, the Albanese torsor, and the Albanese morphism of $V/K$} (although of course the torsor is itself a variety, too).
 
 \begin{rem}\label{R:Pic0T}
  Recall that if $A$ is an abelian variety over $K$ and $P$ is a torsor under $A$ over $K$, then there is a natural isomorphism $A^\vee\stackrel{\iso}{\longrightarrow}\operatorname{Pic}^0_{P/K}$ 
 (\emph{e.g.}, \cite[\S 2.1]{olssonAV}).  
 Moreover, if $A$ and $A'$ are abelian varieties over $K$, and $P$ and $P'$ are torsors under $A$ and $A'$, respectively, then for any $K$-morphism $g:P\to P'$, there is a unique $K$-homomorphism $\phi:A\to A'$ making $g$ equivariant, and moreover, $g(P)$ is a torsor under $\phi(A)$; more precisely, $\phi$ is the composition $A\stackrel{\iso}{\longrightarrow} (\operatorname{Pic}^0_{P/K})^\vee \stackrel{(g^*)^\vee}{\to} (\operatorname{Pic}^0_{P'/K})^\vee \stackrel{\iso}{\longrightarrow}A'$.  
 In particular, in the definition of the Albanese data above, 
there is a unique  $K$-homomorphism $\operatorname{Alb}_{V/K}\to A$ making $g:\operatorname{Alb}^1_{V/K}\to P$ equivariant.  
\end{rem}

 When $V/K$ is equipped with a $K$-point $v:\operatorname{Spec}K\to V$ over $K$,
 then one can define a \emph{pointed Albanese variety and morphism}, by requiring all the maps in the
 previous paragraph to be pointed.  This reduces to the following situation: a
 \emph{pointed Albanese datum for $(V,v)$} is a pair 
  \begin{equation}\label{E:PAlbDat}
 (\operatorname{Alb}_{V/K},\ a_{V/K,v}:(V,v)\to (\operatorname{Alb}_{V/K},\mathbf 0 ))
  \end{equation}
  where $\operatorname{Alb}_{V/K}$ is an abelian
 variety, and $a_{V/K,v}:V\to \operatorname{Alb}_{V/K}$ is a morphism of $K$-schemes
 taking~$v$ to the zero section $\mathbf 0=\mathbf 0_{\operatorname{Alb}_{V/K}}$, which is initial,  meaning that given any
 $K$-morphism $f:V\to A$  to an abelian variety $A/K$, taking $v$ to the zero
 section $\mathbf 0_A$, there is a unique $K$-homomorphism $g:\operatorname{Alb}_{V/K}\to A$ making
 the following diagram of pointed $K$-morphisms commute:
 $$
 \xymatrix{
  (V,v) \ar[r]^<>(0.5){a_{V/K,v}} \ar[rd]_f& ( \operatorname{Alb}_{V/K},\mathbf 0_{\operatorname{Alb}_{V/K}}) \ar@{-->}[d]_{\exists !}^g \\
 & (A,\mathbf 0_A)
 }
 $$

As we noted in the introduction, the existence of Albanese data in the case of complex projective manifolds is classical, while in the more general setting goes back essentially to Serre \cite{serrealb}.  We
        direct the reader to \cite[Thm.~A.1 and p.836]{wittenberg08}
        for an exposition valid over an arbitrary field; the
        assertion there is made for $V/K$ a geometrically integral
        scheme\footnote{Note that Wittenberg uses the term variety for a scheme of finite type over a field \cite[p.807]{wittenberg08}.} of finite type over a field~$K$, although the
        argument holds under the slightly weaker hypothesis that $V$ is a geometrically connected  and 
        geometrically reduced 
        scheme of finite type over $K$:   
 \begin{teo}[Serre]\label{T:Serre-Alb}
  Let $V$ be a geometrically connected and  geometrically reduced scheme
  of finite type over a field $K$.  Then $V$ admits Albanese data, and if $V$ admits
  a $K$-point, then $V$ admits pointed Albanese data. 
  \end{teo}

 \subsection{Necessity of geometric connectedness and geometric reducedness}

We now briefly discuss the hypotheses in \Cref{T:Serre-Alb} that $V$
be geometrically connected and geometrically reduced.  In short, the
geometric connectedness of $V$ is necessary for $V$ to admit 
Albanese data (\Cref{C:GCX}), while the
geometric reducedness of $V$ is not (\Cref{E:NonRedAlb}). 
  The situation is summarized in \Cref{C:Alb-nEx}\ref{C:Alb-1}.

\medskip 
The basic starting point is the following existence result, which states that given an abelian variety and a zero-dimensional scheme, there is a second abelian variety containing the zero dimensional scheme as a closed subscheme, and  which admits no non-trivial homomorphisms from the first abelian variety:

\begin{pro}
\label{P:embed} 
 Given an abelian variety $A/K$ and a finite dimensional $K$-algebra $R$ with each residue field a simple extension of $K$ that is either separable or purely inseparable,  there exists an abelian variety $A'/K$ such that there is a closed embedding of $K$-schemes  $\spec R \hookrightarrow A'$ and such that $\hom(A_{\bar K}, A'_{\bar K}) = 0$.
\end{pro}

The proof of  \Cref{P:embed} is somewhat lengthy, and so to maintain the flow of the ideas in this subsection,  we postpone the proof until \S \ref{S:embed}.  As an immediate consequence of  \Cref{P:embed}, we have the following:

\begin{teo}
  \label{C:integral} 
  Let $V$ be a scheme of finite type over a field $K$, and suppose there exists 
  a nontrivial finite dimensional $K$-algebra $R$ with $R\subseteq \Gamma(V, \mathcal O_V)$. 
  Then given an  abelian variety $A/K$,  a torsor $P/K$  under
$A$, and a $K$-morphism $f:V \to P$, there exists a torsor $P'/K$ under an abelian
variety $A'/K$ and a $K$-morphism $f':V \to P'$ that does not factor through
$f$.  In particular, $V$ does not admit an Albanese datum.
\end{teo}

\begin{proof} The first claim is that $R$ contains a nontrivial finite dimensional  $K$-subalgebra $R'\subseteq R$ with residue fields that are separable extensions of $K$ or simple purely inseparable extensions of $K$.  Indeed,  $R$  being Artinian is a direct sum $R = \oplus_{i=1}^c
  R_i$ of finite local $K$-algebras $(R_i, \mathfrak m_i)$.  For each
  $i$, we may and do choose a sub-algebra $R'_i$ such that, if the
  residue field $\kappa_i:= R_i/\mathfrak m_i$ is a nontrivial
  extension of $K$, then $\kappa'_i := R'_i/(\mathfrak m_i\cap R_i)$ is
  a nontrivial extension which is either separable or simple and
  purely inseparable.  Indeed, if $\kappa_i$ is purely inseparable,
  let $K_i \subseteq \kappa_i$ be a sub-$K$-extension of degree
  $\operatorname{char}(K_i)$, and otherwise let $K_i$ be the separable
  closure of $K$ in $\kappa_i$.  In either case, let $R_i =
  \varpi_i^{-1}(K_i)$, where $\varpi_i:R_i \to \kappa_i$ is the
  projection.  If $\kappa_i = K$ then we simply set $R'_i = R_i$.  
  Finally we let $R' = \oplus_{i=1}^c R'_i$.  
    
  By virtue of \Cref{P:embed}, let $A'/K$ be an abelian variety such that $\operatorname{Hom}(A,A')=
  0$,  and such that there is a closed immersion $\operatorname{Spec}R'\hookrightarrow A'$. 
Since $V\to \operatorname{Spec}\Gamma(V,\mathcal O_V)$ is
scheme-theoretically surjective (for any ring $S$ we have
$\operatorname{Hom}(V,\operatorname{Spec}
S)=\operatorname{Hom}(\operatorname{Spec}\Gamma(V,\mathcal
O_V),\operatorname{Spec} S)$) and since the inclusion $R'\subseteq
\Gamma(V,\mathcal O_V)$ induces a scheme-theoretic surjection (for
affine schemes the scheme-theoretic image is determined by the
factorization of a ring homomorphism into a surjection followed by an
inclusion), we have that $V\to \operatorname{Spec}R'$ is
scheme-theoretically surjective.

Now let  $f'$  be the composition $f':V\twoheadrightarrow
\operatorname{Spec}R' \hookrightarrow A'$. We obtain a diagram
  \begin{equation}
    \label{E:impossfact}
  \xymatrix{
    V \ar[dr]^{f'} \ar[d]_f \ar@{->>}[r]^{} & \operatorname{Spec}R'\ar@{^{(}->}[d] \\
    P \ar@{-->}[r]_{g} & A',
  }
  \end{equation}
so that if we had a 
  factorization $f' = g \circ f$, as indicated by the dashed arrow, 
  then the  morphism $g$ would be an equivariant morphism over a $K$-homomorphism $\phi:A
  \to A'$ of abelian varieties.  The hypothesis that $\operatorname{Hom}(A,A')=0$, would force $\phi$ to  be the trivial
  map, so that $g$ would be constant, with image a $K$-point of $A'$.
  The commutativity of the diagram would then imply
  $\operatorname{Spec}R'\cong \operatorname{Spec}K$, which we have
  assumed is not the case.
\end{proof}

\begin{cor}[Geometric connectedness is necessary]\label{C:GCX}
Suppose that $V$ is a scheme of finite type over a field $K$, and $V$ fails to be geometrically connected.  Then $V$ does not admit an Albanese datum.
\end{cor}

\begin{proof}
  By Corollary \ref{C:integral}, it suffices to show that $\Gamma(V,
  \calo_V)$ contains a nontrivial finite $K$-algebra.
Write $V = \coprod_{i=1}^c V_i$ as a disjoint union of  connected
$K$-schemes. Then there are idempotents $e_i \in \Gamma(V, \mathcal
O_V)$ such that $\Gamma(V_i, \mathcal O_{V_i}) = \Gamma(V_i, \mathcal
O_V) = e_i \Gamma(V, \mathcal O_V)$.  Therefore, if $V$ is
disconnected, then $c>1$ and
$R := \bigoplus_{i=1}^c K e_i  \iso K^{\oplus c}
\subseteq \Gamma(V, \calo_V)$ is a nontrivial finite $K$-algebra. 

Otherwise, assume $V$ is connected, but geometrically disconnected.  Let $L/K$ be a finite Galois
extension such that $V_L$ is
disconnected (\emph{e.g.}, \cite[Prop.~5.53]{GW20}).  As before, write $V_L =
\coprod_{i=1}^{c} W_i$ as a disjoint union of $c>1$ 
connected $L$-schemes, and let $e_1, \dots, e_c \in \Gamma(V_L, \mathcal O_{V_L})$
be the corresponding idempotents.  Then $\operatorname{Gal}(L/K)$ permutes the $e_i$, and
acts transitively because $V$ itself is connected.  
As before, we have $\bigoplus_{i=1}^cLe_i\cong L^{\oplus c}\subseteq \Gamma(V_L,\mathcal O_{V_L})$.  In fact, letting $H \subseteq \operatorname{Gal}(L/K)$ be the stabilizer of $e_1$, 
then  $c=|\operatorname{Gal}(L/K)|/|H|$, and we can enumerate the components of $V_L$ by the cosets  $g_1H,\dots, g_cH$ for some elements $g_1,\dots,g_c\in \operatorname{Gal}(L/K)$.  
In this notation, we can take 
$e_1 = (1,0,\dots,0)$, $e_2=(0,1,0,\dots,0)$, etc., and then the  action of $\gal(L/K)$ on $L^{\oplus c}$ is given by  $$g\cdot (\ell_{g_1H}, \ell_{g_2H},...,\ell_{g_nH}) = (g\cdot \ell_{g^{-1}g_1H}, g\cdot \ell_{g^{-1}g_2H},\dots,g\cdot \ell_{g^{-1}g_nH});$$
in other words, $\operatorname{Gal}(L/K)$ permutes the components according to  its action on the  cosets of $H$,  and then acts on the entries according to the action of the Galois group on $L$.
 
Now, because $\operatorname{Gal}(L/K)$
does not fix $e_1$, it follows that $H$ is a proper subgroup of $\operatorname{Gal}(L/K)$, and its fixed field $L^H$ 
satisfies $[L^H:K] = |\operatorname{Gal}(L/K)|/\abs H > 1$.  From the description of the action of $\operatorname{Gal}(L/K)$ on  $L^{\oplus c}$, it follows that there is a copy of $L^H$ in $\bigoplus ^cL\subseteq \Gamma(V_L,\mathcal O_{V_L})$ given by
$$
\ell\mapsto (g_1\cdot \ell, g_2\cdot \ell, \dots,g_n\cdot \ell),
$$
which is, by construction, invariant under the action of $\operatorname{Gal}(L/K)$.
           Thus $\Gamma(V, \mathcal O_V)$, being the $\operatorname{Gal}(L/K)$-invariants of $\Gamma(V_L,\mathcal O_{V_L})$,
contains a  ring isomorphic to the finite nontrivial $K$-algebra~$L^H$.   
\end{proof}

\begin{rem}[Geometric connectedness of $\Gamma(V,\mathcal O_V)$ is necessary]\label{R:geomconn} We note that for  a scheme $V$ of finite type over a field $K$, since $V$ is disconnected if and only if $\spec \Gamma(V,\mathcal O_V)$ is disconnected, we have that $V$ 
 fails to be geometrically connected if and only if  $\operatorname{Spec} \Gamma(V,\mathcal O_V)$ fails to be geometrically connected.
        \end{rem}

We now turn our attention to the geometric reducedness hypothesis in \Cref{T:Serre-Alb}, which is more subtle.   
We first observe that since there are non-reduced schemes $V$ of finite type over a field $K$  such that $\Gamma(V,\mathcal O_V)$ does not admit 
  any non-trivial finite $K$-subalgebra $R$ (see \emph{e.g.}, Example \ref{E:NonRedAlb}), the proof of Corollary \ref{C:GCX} cannot be used to rule out the existence of an Albanese in the case where $V$ is non-reduced.    In fact, there are non-reduced schemes that admit Albaneses: 

\begin{exa}[Non-reduced scheme with an Albanese]\label{E:NonRedAlb}
Let $K$ be a field,  $H\subseteq \mathbb P^2_K$ be a line, and take $V=2H\subseteq \mathbb P^2_K$.  
Then $\operatorname{Alb}_{V/K}=\operatorname{Spec}K$, $\operatorname{Alb}^1_{V/K}=\operatorname{Spec}K$, and $a:V\to \operatorname{Alb}^1_{V/K}$ is the structure map (of $V$ as a $K$-scheme).
Indeed, observe first that 
taking the long exact sequence in cohomology associated to $0\to \mathcal O_{\mathbb P^2_K}(-2H)\to \mathcal O_{\mathbb P^2}\to \mathcal O_V\to 0$, one has that $\Gamma(V,\mathcal O_V)=K$.    Therefore, as under the standard identification $\operatorname{Hom}(V,\operatorname{Spec}R)=\operatorname{Hom}(R,\Gamma (V,\mathcal O_V))$ for a ring $R$, every morphism $V\to \operatorname{Spec}R$ factors through the natural morphism $V\to \operatorname{Spec}\Gamma(V,\mathcal O_V)$, then for any scheme-theoretically surjective morphism $V\to \operatorname{Spec}R$, we have $R=K$.   
 Now, since $V_{\operatorname{red}}=\mathbb P^1_K$, then given any $K$-morphism $V\to P$ to a torsor $P$ under an abelian variety $A/K$, the composition $\mathbb P^1_K=V_{\operatorname{red}}\hookrightarrow V\to P$ has set-theoretic image a $K$-point of $P$.  Thus the scheme-theoretic image of $V$ in $P$ is an affine scheme $Z=\operatorname{Spec}R$ where $R$ is a finite $K$-algebra.  Thus $Z=\operatorname{Spec}K$ and we are done.
\end{exa}

 Nevertheless, \Cref{C:integral} does give examples of non-reduced schemes that do not admit Albaneses:

\begin{cor}[Reducedness of $\Gamma(V,\mathcal O_V)$ is necessary] \label{C:RX}
Suppose that $V$ is a scheme of finite type over a field $K$, and $\operatorname{Spec}\Gamma(V,\mathcal O_V)$ fails to be reduced. Then $V$ does not admit an Albanese datum.
\end{cor}

\begin{proof} If  $\operatorname{Spec}\Gamma(V,\mathcal O_V)$ is non-reduced, then there exists a nilpotent element  $r\in \Gamma(V_L,\mathcal O_{V_L})$ 
such that $r^n\ne 0$ and $r^{n+1}=0$ for some natural number $n$.
  Then consider the subring $K[x]/(x^{n+1})\cong R:= K[r]\subseteq \Gamma(V,\mathcal O_{V})$.  
   We conclude using \Cref{C:integral}.
\end{proof}

\begin{exa}[Non-reduced scheme with no Albanese]\label{E:NonRedAlbX}
Let  $K$ be a field, let $L/K$ be a nontrivial finite purely inseparable
field extension,  
 let $Y$ be any scheme of finite type over $L$, and let $V=Y\times_KL$.  Then $\Gamma(V, \mathcal O_V)$ contains $L\otimes_KL$ and thus has nontrivial nilpotents; by  \Cref{C:RX}, $V/L$ does not admit an Albanese datum.  
Similarly, if $Y$ is any scheme of finite type over $K$ and $V=Y_{ K[\epsilon]/(\epsilon^2)}$, then $V/K$ does not admit an Albanese datum. 
As a consequence, in contrast with Example \ref{E:NonRedAlb} where we saw that the nonreduced scheme $V=2H\subseteq \mathbb P^2_K$ admits an Albanese datum over $K$,  we have that the nonreduced scheme $\mathbb P^1_{K[\epsilon]/(\epsilon^2)}$ does \emph{not} admit an Albanese datum over~$K$, even as the reductions of both schemes are isomorphic to $\mathbb P^1_K$, assuming $H$ is chosen with a $K$-point. 
 \end{exa}

\begin{exa}[Reduced but geometrically non-reduced scheme with no
  Albanese]\label{E:NonGeomRedAlbX}
Let $K$ be a field, let $L/K$ be a nontrivial finite purely
inseparable field exension, and let $Y/L$ be a smooth irreducible
variety.  As a $K$-scheme, $Y$ is reduced but not geometrically
reduced; and the presence of the nontrivial finite $K$-algebra $L$
 in
$\Gamma(Y, \mathcal O_Y)$ prevents $Y$ from admitting an Albanese
datum.
\end{exa}

Allowing $Y$ to be affine in \Cref{E:NonGeomRedAlbX} raises the
possibility that the geometric reducedness of $\spec \Gamma(V,
\mathcal O_V)$ is necessary.  However, we have:

\begin{exa}[\emph{Geometric} reducedness of $\Gamma(V,\mathcal O_V)$ is \emph{not} necessary] \label{E:GRX}
We use a well-known example due to Maclane \cite[p.384]{maclane}, which seems to be used frequently as an example of a geometrically nonreduced variety with interesting properties.  Let  $K=\mathbb F_p(t_1,t_2)$, let $S=K[x_1,x_2]/(t_1x_1^p+t_2x_2^p-1)$, and   define $V:=\operatorname{Spec}S$.  Then $V$ is integral and geometrically connected, but not geometrically reduced.  Indeed, setting $L= \mathbb F_p(t_1^{1/p},t_2^{1/p})$, we have $V_L=\operatorname{Spec}L[x_1,x_2]/((t_1^{1/p}x_1+t_2^{1/p}x_2-1)^p)$. 
One can check that $K$ is algebraically closed in $S$ (see,
\emph{e.g.},  \cite[Exa.~6.15, p.20]{vojta}), so that $S$, being
reduced, admits no non-trivial finite dimensional $K$-subalgebras
$R\subseteq S$.  Differently put, one cannot use \Cref{C:integral} to try to show that $V$ does not admit an Albanese datum. In fact,   we claim that the structure morphism $V\to \operatorname{Spec}K$ is an Albanese datum.  In other words, there are no nontrivial morphisms $V\to P$ to a torsor under an abelian variety over $K$. 
This follows from the fact that the reduction of $V_L$ is a rational curve.  More precisely,   given a morphism $V\to P$ to a torsor under an abelian variety over $K$, if the image were zero dimensional, then since
 $S$ admits no non-trivial finite dimensional $K$-subalgebras $R\subseteq S$, the image of $V$ in $P$ would have to be isomorphic to  $\operatorname{Spec} K$.  If  the image of $V$  were $1$-dimensional, then after base change to an algebraic closure $\bar K$, and considering the reduction of $V_{\bar K}$, one would have a non-trivial map from a rational curve to an abelian variety, which is not possible.  Since $\dim V=1$, we are done. 
\end{exa}

We summarize the situation in the following corollary, including the relation to Schr\"oer \cite[Thm.~p.2]{schroeralbanese}, which has the additional separated hypothesis:
  \begin{pro}\label{C:Alb-nEx}
Let $V$ be a scheme of finite type over a field $K$.  
\begin{enumerate}[label={(\roman*)}]

\item 
If $V$ admits an Albanese datum, then \label{C:Alb-1}
\begin{enumerate}[label={(\alph*)}]
\item $V$ is geometrically connected, \label{C:Alb-a}
\item $\operatorname{Spec}\Gamma(V,\mathcal O_V)$ is geometrically connected,\label{C:Alb-b}
\item $\operatorname{Spec}\Gamma(V,\mathcal O_V)$ is reduced, \label{C:Alb-c}
\item $K$ is algebraically closed in $\Gamma(V,\mathcal O_V)$, \label{C:Alb-d} 
\end{enumerate}
\item (Schr\"oer) If, moreover, $V$ is separated, then the converse holds. More precisely, for a separated scheme $V$ of finite type over a field $K$, one has that  $V$ admits an Albanese  datum if and only if $\operatorname{Spec}\Gamma(V,\mathcal O_V)$ is connected and reduced, and $K$ is algebraically closed in $\Gamma(V,\mathcal O_V)$.  \label{C:Alb-2}
\end{enumerate}
\end{pro}

\begin{rem}
Note that if 
  \ref{C:Alb-d} holds in \Cref{C:Alb-nEx}\ref{C:Alb-1}, then $V$ (resp.~$\operatorname{Spec}\Gamma(V,\mathcal O_V)$) connected implies $V$ (resp.~$\operatorname{Spec}\Gamma(V,\mathcal O_V)$) is geometrically connected. 
\end{rem}

\begin{proof}Assuming $V$ admits an Albanese datum,  \ref{C:Alb-a} and
  \ref{C:Alb-b} are \Cref{C:GCX} and  \Cref{R:geomconn}. \ref{C:Alb-c}
  is \Cref{C:RX}.  The assertion \ref{C:Alb-d} follows immediately from
  \Cref{C:integral}, since if $K$ is not algebraically closed in
  $\Gamma(V,\mathcal O_V)$, then $\Gamma(V,\mathcal O_V)$ contains a
  finite nontrivial extension field of $K$.

Conversely, assume that  $V$ is a separated scheme $V$ of finite type over a field $K$,    $\operatorname{Spec}\Gamma(V,\mathcal O_V)$ is connected and reduced, and $K$ is algebraically closed in $\Gamma(V,\mathcal O_V)$.
The conclusion that $V$ admits an Albanese datum is then is due to Schr\"oer \cite[Thm.~p.2]{schroeralbanese}, after one observes that for a separated scheme $V$ of finite type over a field $K$, with $\operatorname{Spec} \Gamma (V,\mathcal O_V)$ connected and reduced, then $V$ is naturally endowed with a scheme structure over the essential field of constants $K'$ \cite[p.19]{schroeralbanese}, which is by construction finite over $K$; \emph{i.e.}, there is a factorization  $V \to \operatorname{Spec}\Gamma(V,\mathcal O_V)\to  \operatorname{Spec} K'\to  \operatorname{Spec} K$.
Therefore, if $K$ is algebraically closed in $\Gamma(V,\mathcal O_V)$, then the essential field of constants for $V$ is $K$. The assertion is then exactly the statement of \cite[Thm.~p.2]{schroeralbanese}.
\end{proof}

\subsection{Embedding zero-dimensional schemes in abelian varieties}\label{S:embed}

While \Cref{P:embed}  is well suited to proving 
\Cref{C:integral}, the following stronger existence result seems easier to verify:

\begin{pro}
\label{P:rkcondremix}
\begin{enumerate}[label={(\alph*)}]
  Let $L/K$ be a finite simple extension.
       \item If $L/K$ is separable, then there exists a collection of abelian varieties $\st{A_i/K}$ of unbounded dimension such that, for each $i$, \label{P:rkcond1}
\begin{itemize}
\item $A_i$ is absolutely simple;
\item $A_i$ has a closed point with residue field $L$; and
\item $\# A_i(K) \ge 2$.
\end{itemize}

\item If $L/K$ is purely inseparable, then there exist a collection of abelian varieties $\st{A_i/K}$  of unbounded dimension and a collection of abelian varieties $\st{B_i/K}$ such that, for each $i$, \label{P:rkcond2}
\begin{itemize}
\item $A_i$ is absolutely simple;
\item $(B_i)_{\bar K}\cong (A_i\times_KA_i)_{\bar K}$ if $\operatorname{char}(K)\ne 2$, and $(B_i)_{\bar K}\cong  (A_i\times_KA_i\times_KA_i)_{\bar K}$ if $\operatorname{char}(K)=2$;
\item $B_i$ has a closed point with residue field $L$; and
\item $\# B_i(K) \ge 2$.
\end{itemize}

\end{enumerate}
\end{pro}

Before proving \Cref{P:rkcondremix}, we explain how it implies \Cref{P:embed}:

\begin{proof}[Proof of \Cref{P:embed} (using \Cref{P:rkcondremix})]
 First assume that $R$ is local, and set $Z=\operatorname{Spec}R$ for simplicity of notation.  
 By assumption $Z\subseteq \operatorname{Spec}K[x_1,\dots,x_n]$ for any sufficiently large $n$, and in particular we may assume  $n>3\dim A$.  Take $A'$ to be an abelian variety from \Cref{P:rkcondremix} with $\dim A'\ge n$, and replacing $n$ with $\dim A'$, we can and will assume that $n=\dim A'$.
  Since $(A')_{\bar K}$ is a product of at most $3$ simple abelian varieties each of which, from our assumptions on~$n$, must  have dimension greater than $\dim A$, we have that $\operatorname{Hom}(A_{\bar K},A'_{\bar K})=0$.  Therefore, we only need to show that $Z\subseteq A'$.  
 
  For this,  using the definition of smoothness, we have a commutative fibered product diagram Zariski locally on $A'$:
 $$
 \xymatrix{
 Z' \ar@{^(->}[r] \ar[d]_{\text{\'et}}& (A',a') \ar[d]_{\text{\'et}}\\
 (Z,z) \ar@{^(->}[r]& (\mathbb A^n_K,a)
 }
 $$
 where we have marked each scheme with its respective $L$-point, having residue field $L$.  
 The $L$-points, and the commutativity of the diagram give an $L$-point we will call $z'$ on $Z'$.  Let $Z''$ be the component of $Z'$ containing $z'$, and consider the pointed scheme $(Z'',z')$.  Note that the residue field of $z'$ must also be equal to $L$.    
Since all the morphisms above induce isomorphisms on the complete local rings (they are \'etale and induce isomorphisms on residue fields \cite[Prop.~17.6.3]{EGAIV4}), and since  $(Z'',z')$ and $(Z,z)$ are affine pointed schemes associated to finite $K$-algebras (which are therefore products of complete local $K$-algebras),  we have that $Z''$ and $Z$  are isomorphic.  This completes the proof in the case where $R$ is local.

In general, $R$, being Artinian, is a product of finitely many local
rings.  Now use the nontrivial $K$-points and a product construction.
(In more detail, if $R \iso \prod_{j=1}^r R_j$ is a product of local
Artin algebras, using the previous paragraph, let $A_j/K$ be an absolutely simple abelian variety
equipped with an embedding $\alpha_j: \spec R_j \to A_j$ whose image
is \emph{not} supported at the identity element.  Let $A' =
\prod A_j$, and let $\iota_j: A_j \to A'$ be the natural embedding.
Define a morphism $\alpha: \spec R \to A$ by $\alpha|_{\spec(R_j)} =
\iota_j \circ \alpha_j$; then $\alpha$ is a closed embedding.)
\end{proof}

\begin{proof}[Proof of \Cref{P:rkcondremix} when $K$ is infinite and $L/K$ is a finite  \emph{separable} extension]

We suppose  
$\operatorname{char}(K)\not  = 2$, and indicate the
necessary changes for even characteristic at the end.  Using the
separability hypothesis, choose a
polynomial $f(s) \in K[s]$ such that $L \iso K[s]/f(s)$; note that $f$
is squarefree.

For each $i$, let $h_i(x) \in K[x]$ be a polynomial of degree $i$
which factors completely over $K$, and such that $x h_i(x)$ is
squarefree.
 Let $t$ be a parameter on $\mathbb A^1_K$, and let
$\mathcal C_i \to \mathbb A^1_K$ be the family of curves whose fiber over
$t$ has affine model $y^2 = f(x)x h_i(x)(x-t)$.  Fix $\ell > 3$
invertible in $K$.  The geometric $\bmod \ell$ monodromy of this family
is $\operatorname{Sp}_{2g}(\integ/\ell)$ \cite{hallbigmonodromy}.

Let $K_0 \subseteq K$ be a subfield, finitely generated over the prime
field, such that $\mathcal C_i \to \mathbb A^1_K$ descends to a model over
$K_0$. Since $K_0$ is finitely generated over $\rat$ (if
$\operatorname{char}(K)=0$) or $\ff_p(s)$ (if
$\operatorname{char}(K)=p>0$), $K_0$ is Hilbertian.  By Hilbert's
irreducibility theorem, there exists some $t_0 \in \mathbb A^1(K_0) \subset
\mathbb A^1(K)$ such that, for $C_i := \mathcal C_{i,t_0}$, the image of $\gal(K_0)$ acting on $H^1(C_{i, \bar K_0},\integ/\ell)$ contains $\operatorname{Sp}_{2g}(\integ/\ell)$.  In particular, let $A_i = \operatorname{Jac}(C_{i,t_0})/K$; a standard argument then shows $\End(A_{i,\bar K}) \iso \integ$.  (Briefly, for 
  group-theoretic reasons, since $\ell>3$ and since the image of
  $\gal(K_0)$ in $\operatorname{Aut}(H^1(C_{i,\bar K_0},\integ_\ell))$
  is an $\ell$-adically closed group which contains a subgroup surjecting onto
  $\operatorname{Sp}_{2g}(\integ/\ell)$, it contains all of
  $\operatorname{Sp}_{2g}(\integ_\ell)$.  Replacing $K_0$ with a
  finite extension $K_0'$ replaces the image of $\gal(K_0)$ with a
  subgroup of finite index, but it is still Zariski dense in
  $\operatorname{Sp_{2g,\rat_\ell}}$.  Then
  $\End_{K'_0}(A_i)\otimes\rat_\ell$, being contained in the commutant
  of $\operatorname{Sp}_{2g,\rat_\ell}$ in $\operatorname{Aut}(H^1(C_{i,\bar K_0},\rat_\ell))$, is just $\rat_\ell$, and thus
  $\End_{K'_0}(A_i) \iso \integ$.) 
 We use the base point $(0,0)$ to embed $C_i$ in~$A_i$.

Let $Z_i \subseteq C_i$ be the vanishing locus of the function $y$.  Then $Z_i$ is the spectrum of 
\[
R_i :=\frac{K[x,y]}{(y^2-xh_i(x)(x-t_0), y)} \iso \frac{K[x]}{f(x)} \oplus K^{\oplus i+2} \iso L\oplus K^{\oplus i+2},
\]
and we have $Z_i \hookrightarrow C_i \hookrightarrow A_i$.  In
particular, $A_i$ contains a subscheme isomorphic to $\spec L$, and
$\#A_i(K) \ge i+2 \ge 2$.

In fact, the same argument works if we replace $y^2 = xh_i(x)(x-t)$
with $y^r = x h_i(x)(x-t)$ for any prime $r$ \cite[\S 2]{katzbigmonodromy}.  Briefly, every $K$-rational fiber contains $\spec L\oplus \spec K$ as a closed subscheme;  the monodromy group of the family
contains  a 
special unitary group; Hilbert irreducibility and an $\ell$-adic calculation show that the absolute endomorphism ring of the Jacobian of any
fiber outside a thin set is $\integ[\zeta_r]$; and such a Jacobian is
an absolutely simple abelian variety.
     \end{proof}

We now move to the case of a finite simple purely inseparable extension $L/K$, which we can take to be of the form $L=K[x]/(x^{p^r}-a)$ for some choice of $a\in K$.  A natural approach (for $\operatorname{char}K\ne 2$) would then be to consider the family of hyperelliptic curves 
$ -y + y^2 + xy + (x^{p^r}-a)(x^p-a_1)(x^p-a_2)\cdots (x^p-a_s)(x^p-t)$, or the family  $-ty + y^2 + xy + (x^{p^r}-a)(x^p-a_1)(x^p-a_2)\cdots (x^p-a_s)$, both of which give smooth affine curves containing $\operatorname{Spec}L$ as a closed subscheme. We note that completing the square of the second family gives the family $y^2-\frac{1}{4}(x-t)^2+(x^{p^r}-a)(x^p-a_1)(x^p-a_2)\cdots(x^p-a_s)$.  The complication in this approach is to determine if the general member of the family has large $\bmod \ell$ monodromy; i.e., whether the general curve in the family is absolutely simple.  To avoid this issue, we use an argument we learned from David Grant\,:

\begin{proof}[Proof of \Cref{P:rkcondremix} when $K$ is infinite and $L/K$ is a finite simple \emph{purely inseparable} extension]  As in the previous proof, we reduce to the case where $K$ is finitely generated over the prime field, and again, we explain the case where $p := \operatorname{char}(K)\ne 2$ first.
To begin, we fix for each $i$ the curve $C_i$ over $K$ from the previous proof, which is a smooth projective hyperelliptic curve of appropriately large genus admitting a number of $K$-points, and which has absolutely simple Jacobian $\operatorname{Jac}(C_i)$, which we denote by $A_i/K$.  
 Since $L/K$ is assumed to be simple and purely inseparable, we can take $L$ to be  of the form $L=K[x]/(x^{p^r}-a)$ for some choice of $a\in K$.  In particular, we have $\operatorname{Spec}L\subseteq \mathbb P^1_K$, and by changing $C_i$ (moving the branch points) we may and will assume that $\operatorname{Spec}L$ has support disjoint from the branch locus of the structure map $C_i\to \mathbb P^1_K$.  The pre-image of $\operatorname{Spec}L$ in $C_i$ is a closed subscheme $\operatorname{Spec}N\subseteq C_i$, either consisting of two distinct $L$-points, or   consisting of a single point, in which case  $N/L$ is a degree $2$ extension of fields, separable since  $\operatorname{char}(K)\ne 2$.  
In the former case we can simply take $B_i=A_i\times_KA_i$.  

In the latter case,  
 taking the separable closure of $K$ in $N$, we obtain another subfield $M/K$ of $N$, necessarily of degree $2$ over $K$, giving us a diagram of fields
 \begin{equation*}\label{E:tower0}
 \xymatrix@R=1em{
   N \ar@{-}[d]_{\text{sep}} \ar@{-}[dr]^{\text{insep}} \\
   L\ar@{-}[dr]_{\text{insep}}^<>(0.25){p^r}   & M \ar@{-}[d]^{\text{sep}}_2\\
   &K
  }
\end{equation*}
Using that $[L:K]$ and $[M:K]$ are coprime, we have that $N=LM$, and   $L$ and $M$ are linearly disjoint over $K$ ($N=L\otimes_KM$).
 
Now since $C_i$ is a smooth projective curve with a $K$-point, and therefore embeds in its Jacobian,  we have that $A_i$ admits $\operatorname{Spec}N$ as a closed subscheme (and has a number of $K$-points).   Then $B_i:=\operatorname{Res}_{M/K}((A_i)_M)$ has the property that $(B_i)_{\bar K}\cong (A_i\times_KA_i)_{\bar K}$ (\emph{e.g.}, \cite[Lem.~5]{flynntesta15}), 
  and we claim that $B_i$ admits $\operatorname{Spec}L$ as a closed subscheme (as well as a $K$-point for each  $K$-point of $A_i$). This latter assertion follows from the fact that $\operatorname{Res}_{M/K}\operatorname{Spec}N$ contains $\operatorname{Spec}L$ as a closed subscheme, and the fact that closed immersions are preserved by the Weil restriction \cite[\S 7.6, Prop.~2, p.192]{BLR}.

To see that $\operatorname{Res}_{M/K}\operatorname{Spec}N$ contains $\operatorname{Spec}L$ as a closed subscheme one can argue as follows.
 We have $N=M[x_1,\dots,x_n]/(f_1,\dots,f_m)$.
As for any affine scheme and any finite extension of fields, we can write 
$
{\displaystyle \operatorname {Res} _{M/K}\operatorname{Spec}N}
$
as 
$\operatorname{Spec} K[y_{{i,j}}]/(g_{{l,r}})$, where $y_{i,j}$ ($1\leq i\leq n$, $1\leq j\leq s$) are new variables, and 
$g_{l,r}$ ($1 \leq l \leq m$, $1 \leq r \leq s$) are polynomials in 
$y_{i,j}$ given by taking a basis 
$
e_{1},\dots,e_{s}$ of $M$ over $K$ and setting 
$
x_{i}=y_{{i,1}}e_{1}+\cdots+y_{{i,s}}e_{s}$
 and 
$f_{t}=g_{{t,1}}e_{1}+\cdots + g_{{t,s}}e_{s}$.  In our case, $s=2$. Now since $L$ and $M$ are linearly disjoint and $[N:L]=2$, we have that $e_1,e_2$ form a basis of $N$ over~$L$. So, if we write $\alpha_i$ for the class of $x_i$ in $N$, then we can write $\alpha_i=a_{i,1}e_1+a_{i,2}e_2$ for some elements $a_{i,j}\in L$.  Therefore, by definition, taking $y_{i,j}=a_{i,j}$, we obtain an $L$-point of the Weil restriction. Since not all of the $a_{i,j}$ can be in $K$ (otherwise $L=M$), we have in fact an $L$-point of the Weil restriction with residue field $L$.  This completes the proof in the case where $\operatorname{char}(K)\ne 2$.

In the case where $\operatorname{char}(K)=2$, we replace the family of curves $y^2 = xh_i(x)(x-t)$
with the family $y^3 = x h_i(x)(x-t)$; the rest of the proof goes through identically.
\end{proof}

We now take up the task of dealing with finite fields.

\begin{proof}[Proof of \Cref{P:rkcondremix} when $K$ is finite]
  Let $K = \ff_q$ and $L = \ff_{q^r}$.  There exist absolutely simple
  abelian varieties over $K$ of every dimension, and most of them (in
  particular, at least one in every dimension) have at least two
  $K$-rational points \cite{howezhu02}.  It thus suffices to to assume
  that $r>1$ and show that if $A/K$ is a simple abelian variety then,
  with finitely many exceptions, $A$ has a closed point with residue
  field~$L$.  (In the case $(q,q^r) = (2,4)$, we will prove a slightly
  weaker statement which is still adequate for our purpose.)
  
  Let $X/K$ be any geometrically
irreducible variety.  If $X$ does not contain a closed subscheme
isomorphic to $\spec L$, then every $L$-rational point $P\in
X(L)$ is actually defined over some subfield~$K'$, where $K\subseteq
K' \subsetneq L$.  It suffices to consider points defined over
\emph{maximal} proper subfields of~$L$.  Crudely estimating, we have
the criterion that if
\[
  \# X(\ff_{q^r}) > \sum_{\ell | r} \# X(\ff_{q^{r/\ell}})
\]
(where $\ell$ ranges over prime divisors of $r$), then $X$ has a
closed subscheme isomorphic to $\spec L$.

Let $A/K$ be an abelian variety of dimension $g$.  Weil's theorem on
the eigenvalues of Frobenius of an abelian variety easily yields, for
any extension $\ff_{q^d}$ of $\ff_q$, that
\[
(q^d + 1 - {2\sqrt{q^d}})^g \le \# A(\ff_{q^d}) \le (q^d + 1 +
{2\sqrt{q^d}})^g.
\]
(In fact, each occurrence of ${2\sqrt{q^d}}$ can be replaced by $\floor{2\sqrt{q^d}}$ \cite[\S 1]{aubryhalouilachaud13}.)
In particular, let $\ell_0$ be the smallest prime divisor of $r$.  We have the coarse estimate
\begin{align*}
  \frac{\# \bigcup_{K' \subsetneq L} A(K')}{\#A(L)} &  \le
                                                      \frac{\sum_{\ell|r}
                                                      \#
                                                      A(\ff_{q^{r/\ell}})}{\#A(\ff_{q^r})}\\
   & \le \frac{\left(\# \st{ \ell: \ell|r}\right) (q^{r/\ell_0} + 1 +
     {2 q^{r/2\ell_0}})^g}{(q^r+1-{2\sqrt{q^r}})^g}\\
   & \le
     \log(r) \left(
\frac{q^{r/\ell_0} + 1 +
     {2 q^{r/2\ell_0}}}{q^r+1-{2\sqrt{q^r}}}
     \right)^g.
\end{align*}
For sufficiently large $g$, this quantity is less than one,
\emph{unless} $$(q^{r/\ell_0}, q^r) \in \{(2,4),(3,9),(4,16),(2,8)\}.$$

Consider one of these remaining cases.  Then $\ff_{q^{r/\ell_0}}$ is
the unique maximal proper subfield of~$\ff_{q^r}$.  If the abelian
variety $A/\ff_q$ fails to have a closed subscheme isomorphic to
$\spec \ff_{q^r}$, then $A(\ff_{q^{r/\ell_0}}) = A(\ff_{q^{r}})$.
Except for the case $(q^{r_{\ell_0}},q^r) = (2,4)$, this cannot happen
if $A$ is simple of dimension at least three.  Indeed, the case
$(q^{r_{\ell_0}},q^r) = (2,8)$ literally follows from
\cite[Lem.~3.1]{kedlayarelclassoneI}, while the other two cases follow
from its proof and \cite[Lem.~2.1(b)]{kedlayarelclassoneI}.  

We now
address the remaining case $K = \ff_2$ and $L= \ff_4$ by adapting
Kedlaya's argument to our needs.  Assume that $A/K$ is an absolutely simple ordinary
abelian variety of dimension $g$ with $\#A(K) \ge 2$.  (Again, this is
possible by \cite{howezhu02}.)  Let $A'$ be its nontrivial quadratic
twist; it, too, is absolutely simple.  We will show that if $A(K) =
A(L)$, then $A'(K) \subsetneq A'(L)$.  Let $B = \res_{L/K}(A_L)$; then
$B$ is isogenous to $A\times_K A'$.  Since $B(K) = A(L)$, if $A(K) =
A(L)$, then $\#A'(K)  = 1$.  For $g$ in the complement of a thin set
of natural numbers -- in particular, for infinitely many $g$ -- this uniquely determines the isogeny class of $A'$
\cite[Lem.~2.1(c)]{kedlayarelclassoneI}.   If $A'$ \emph{also} has the
property that $A'(K) = A'(L)$, then the quadratic twist $A''$ of $A'$
also satisfies $A''(K) = 1$.  Since $A'' \iso A$, we find in
particular that $A$ and
$A'$ are isogenous; but this is impossible for a simple ordinary
abelian variety (e.g., \cite[Ex.~1.7]{achtercunningham15}). Consequently, at
least one of $A$ and $A'$
admits a subscheme isomorphic to $\spec L$. 
\end{proof}

As is clear from the proofs above, given a simple extension of fields $L/K$ one can quickly write down an abelian variety with $\operatorname{Spec}L$ as a closed subscheme.  The difficulty is finding such an abelian variety that is absolutely simple (or, whose base change to the algebraic closure is a product of simple abelian varieties of sufficiently large dimension).  In order to shorten the proofs above, one might hope that given any abelian variety $A/K$, and any simple extension of fields~$L/K$, 
there exists a closed subscheme of $A$ isomorphic to $\operatorname{Spec}L$.  The following example shows that this is not the case\,:

\begin{exa}
Let $E$ be the elliptic curve over $\mathbb F_2$ with affine model $y^2+y = x^3+x^2$.  One can check that $\#E(\mathbb F_2) = \#E(\mathbb F_4)=5$. In particular, there is no point of $E$ with residue field $\mathbb F_4$.
\end{exa}

 \section{Base change for Albanese varieties}\label{S:BC}

  Let $V$ be a geometrically connected  and geometrically reduced scheme
  of finite type over a field $K$, and let $(\alb_{V/K}, \
 \alb^1_{V/K},a_{V/K} )$ be Albanese data for $V$ \eqref{E:AlbDat}.  Recall that this includes the Albanese morphism 
  $$
\xymatrix{
  V \ar[r]^<>(0.5){a_{V/K}}&  \operatorname{Alb}^1_{V/K}
  }
  $$
  to the Albanese torsor.  
 If $L/K$ is any field extension, then after base change
 we obtain a diagram 
   \begin{equation}\label{E:beta1}
\xymatrix@C=4em{
  V_L \ar[r]^<>(0.5){a_{V_L/L}} \ar@{=}[d]&  \operatorname{Alb}^1_{V_L/L} \ar@{-->}[d]^{\beta^1_{V,L/K}}\\
  V_L \ar[r]^<>(0.5){(a_{V/K})_L}&  (\operatorname{Alb}^1_{V/K})_L\\
  }
  \end{equation}
 where $\beta^1_{V,L/K}$ is induced by the 
   the universal property of the Albanese.
       Via the dual of the pull-back morphism on line bundles (see \Cref{R:Pic0T}), this is  equivariant with respect to a base change morphism of abelian varieties
  \begin{equation}\label{E:beta}
 \xymatrix{\beta_{V,L/K}: \alb_{V_L/L} \ar[r] &(\alb_{V/K})_L}.
 \end{equation}
A diagram similar to \eqref{E:beta1} holds in the pointed case, as well.

We say that the Albanese data  of $V$ is \emph{stable under (separable) base change of
  field}  if the Albanese data exists and $\beta_{V,L/K}$ and $\beta^1_{V,L/K}$  are isomorphisms  for all
 (separable) field extensions $L/K$.  Note that in particular, this means that
 $((\alb_{V/K})_L, \
 (\alb^1_{V/K})_L,(a_{V/K})_L )$ gives Albanese data for $V_L$.

 There is an analogous notion for pointed Albanese data
 \eqref{E:PAlbDat} to be \emph{stable under (separable) base change of
   field}.  A before, when this holds, the suitable base
 change of pointed Albanese data is again pointed Albanese data.

Initially, we remark that in the unpointed case,   $\beta_{V,L/K}$ is an
 isomorphism if and only if  $\beta^1_{V,L/K}$ is an isomorphism:

 \begin{lem}
  \label{L:betavsbeta1}
  Let $V/K$ be a geometrically connected and geometrically reduced  
  scheme of finite type over a field $K$, and let $L/K$ be an extension of fields.
  Then $\beta_{V,L/K}$ is an isomorphism if and only if $\beta^1_{V,L/K}$ is an
  isomorphism.
\end{lem}

 \begin{proof}
  On one hand, let $T$ be a torsor under an abelian variety $A$ over $K$; then $A\cong (\operatorname{Pic}^0_{T/K})^\vee$ (\Cref{R:Pic0T}).
     Consequently, if $\alb^1_{V_L/L}$ and
  $(\alb^1_{V/K})_L$ are isomorphic via $\beta^1_{V,L/K}$, then so are $\alb_{V_L/L}$ and
  $(\alb_{V/K})_L$ via $\beta_{V,L/K}$.

  On the other hand, suppose $\beta_{V,L/K}$ is an isomorphism.  Then
  $\beta^1_{V,L/K}$ is a nontrivial map of torsors over an isomorphism of abelian
  varieties.  Since $\beta_{V,L/K}$ and $\beta^1_{V,L/K}$ agree up to translation after base change to the algebraic closure of $L$, $\beta^1_{V,L/K}$ is an isomorphism. 
 \end{proof}

 The Raynaud example (Example~\ref{E:badbc}) shows that, in general, Albanese varieties are not stable under base change of
 field.
 There are two possible issues to focus on in this example.  First, the
 variety $G/K$ is not proper, and second, the extension $L/K$ is not
 separable.  Regarding the former, it has been understood that if one
 assumes $V$ is proper, then the Albanese variety is stable under base change:

 \begin{teo}[Grothendieck--Conrad]\label{T:GrConBC-1}
  Let $V$ be a  \emph{proper} geometrically connected   and geometrically reduced 
  scheme  over a field $K$.  Then  Albanese data for $V$ \eqref{E:AlbDat} is stable under base change
  of field, and if $V$ admits a $K$-point $v$, then  pointed Albanese
 data for $(V,v)$ \eqref{E:PAlbDat} is stable
  under base change of field. \qed
 \end{teo}

 \begin{rem}[References for Theorem \ref{T:GrConBC-1}]\label{R:GrCon}
  Recall that Grothendieck provides an Albanese torsor (resp.~pointed Albanese variety) for any proper geometrically connected and 
  geometrically normal  scheme $V$ (resp.~pointed  proper geometrically connected and 
  geometrically normal  scheme $(V,v)$) over a field $K$ in the following way.
  As $V/K$ is proper and geometrically normal, one has that
  $\operatorname{Pic}^0_{V/K}$ is proper \cite[Thm.~VI.2.1(ii)]{FGA}; then by
  \cite[Prop.~VI.2.1]{FGA}, one has that $(\operatorname{Pic}^0_{V/K})_{\operatorname{red}}$ is a
  group scheme (\emph{i.e.}, without the usual hypothesis that $K$ be perfect and the group scheme be smooth).
  It then follows from \cite[Thm.~VI.3.3(iii)]{FGA} that
  $((\operatorname{Pic}^0_{V/K})_{\operatorname{red}})^\vee$
   is an Albanese variety, and using that $V/K$ is geometrically connected,  that there exists an Albanese torsor.
  Conrad has generalized Grothendieck's argument to show that any proper geometrically connected and  
  geometrically reduced  
   scheme $V$ over a field $K$ admits an
  Albanese torsor,  and  a pointed Albanese variety if $V$ admits a $K$-point. For lack of a better reference, we direct the reader to
  \cite[Thm.]{ConradMathOver}.
  His argument is to show that the Albanese variety is the dual abelian variety to the
  maximal abelian subvariety of the (possibly non-reduced and non-proper) Picard
  scheme  $\operatorname{Pic}_{V/K}$.  Grothendieck's theorem can then be
  summarized in this context by saying that his additional hypothesis that $V$ be
  geometrically normal implies that the maximal abelian subvariety of
  $\operatorname{Pic}_{V/K}$ is $\operatorname{Pic}^0_{V/K}$.
That Grothendieck's and Conrad's Albanese varieties are stable under arbitrary field
  extension 
  is  \cite[Thm.~VI.3.3(iii)]{FGA} and \cite[Prop.]{ConradMathOver},
  respectively.  In fact, the Albanese variety enjoys an even
                stronger universal property; see \S \ref{S:univ} below.
 \end{rem}

  While the hypothesis in the theorem that $V$ be geometrically connected is necessary (\Cref{C:GCX}), we point out here that it is possible for geometrically non-reduced schemes to admit Albanese data that is stable under base change of field:

  \begin{exa}[Albanese base change for a non-reduced scheme]
\label{E:NonRedAlbBC}  Let $V$ be the non-reduced scheme defined in Example \ref{E:NonRedAlb}.  Then the Albanese torsor (and the pointed Albanese variety) of $V$ is stable under base change of field.
\end{exa}

 The second potential difficulty in Example \ref{E:badbc}, namely, the
 inseparability of the field extension $L/K$, shows that in the
 absence of properness, something like the 
 separability hypothesis in Theorem \ref{T:main} is necessary.

 \section{Extensions of fields}

 We briefly detour  from our development of the Albanese to gather
 some results on separable and primary extensions of fields.

\subsection{Separable extensions}\label{S:SepExt}

The following elementary results on separable extensions will
ultimately be used to extend the standard \Cref{L:albBCautLK} below 
 to arbitrary separable extensions (as opposed to separable algebraic extensions).  
 
For clarity with the terminology,
 we recall that a (not necessarily algebraic) field extension $L/K$ is separable if for every extension of fields  $M/K$, one has that $M\otimes _KL$ is reduced.  
 Setting $p$ to be the characteristic exponent,  this is equivalent 
    to the condition that $L^p$ and $K$ be linearly disjoint over $K^p$; \emph{i.e.}, that the natural map $L^p\otimes_{K^p}K\to L^pK$ be injective \cite[Rem., p.V.119]{bourbakifields}.
 We say that a field $K$ is separably closed if it admits no separable \emph{algebraic} field extensions.  We say an extension  of fields $L/K$ is  purely inseparable
  if for every $x\in L$, there is an integer $n$ such that $x^{p^n} \in K$, or equivalently~\cite[Prop.~13, p.V.42]{bourbakifields}, if it is an algebraic extension and there are no  nontrivial separable subextensions.  
 
 Note that if $\operatorname{char}(K)=0$ then any field extension of $K$ is separable \cite[Thm.~1, p.V.117]{bourbakifields}.   In general, any field extension $L/K$ 
  factors as $L/L'/K$ with $L'/K$ separable and $L/L'$ purely inseparable; indeed, taking any transcendence basis $T$ for $L/K$ \cite[Thm.~1, p.V.105]{bourbakifields}, one has $K(T)/K$ is separable \cite[Prop.~6, p.V.116]{bourbakifields}, and then the algebraic extension $L/K(T)$ factors as a separable extension $L'/K(T)$ followed by a purely inseparable extension $L/L'$.  
  Here we are using that the composite of two separable extensions is separable \cite[Prop.~9, p.V.117]{bourbakifields}.
  
 \begin{lem}
  \label{L:autLKsepclosed}
  Let $\Omega/k$ be an extension of separably closed fields. Then
  $\Omega/k$ is separable if and only if $\Omega^{\aut(\Omega/k)} = k$.
 \end{lem}

 \begin{proof}
   Without any assumptions on $K$, if $\Omega/K$ is any field extension, then $\Omega$ is separable
  over $\Omega^{\operatorname{Aut}(\Omega/K)}$; see \emph{e.g.} \cite[\S 15.3,
  Prop.~7]{bourbakifields}. In particular, if $\Omega^{\operatorname{Aut}(\Omega/K)} =
  K$, then $\Omega/K$ is separable.

  Conversely, assume that $\Omega/k$ is a separable extension of
  separably closed fields. Since $k$ is separably closed and since any
  sub-extension $\Omega/K/k$ satisfies $K/k$ separable \cite[Prop.~8, p.V.116]{bourbakifields}, in order to show
  that $\Omega^{\operatorname{Aut}(\Omega/k)} =k$, it is enough to show
  that $\Omega^{\operatorname{Aut}(\Omega/k)} /k$ is algebraic. Let
  $\alpha \in \Omega$ be a transcendental element over $k$.  Since
  $\alpha$ extends to a transcendence basis of $\Omega/k$, the map
  $\alpha \mapsto \alpha+1$ extends to an automorphism of $\Omega$ which
  fixes $k$.  Consequently, no element of $\Omega$ transcendental over
  $k$ is fixed by all of $\aut(\Omega/k)$, and
  $\Omega^{\operatorname{Aut}(\Omega/k)} /k$ is algebraic, as desired.
 \end{proof}

 Recall that if $L'/L/K$ is a tower of field extensions, then on the one hand,  
if $L'/L$ is separable and $L/K$ is separable, then $L'/K$ is separable \cite[Prop.~9, p.V.117]{bourbakifields}.   
 On the other hand, if  $L'/K$ separable,
 then $L/K$ is separable \cite[Prop.~8, p.V.116]{bourbakifields}, but $L'/L$
 may not be separable (\emph{e.g.},
 $\mathbb{F}_p(T)/\mathbb{F}_p(T^p)/\mathbb{F}_p$). Nevertheless, we have:

 \begin{lem}\label{L:L/KsepCl}
  Suppose that $L/K$ is a separable extension of fields.  Then $L^{\sep}/K^{\sep}$
  is separable.
 \end{lem}

 \begin{proof} 
           In characteristic $0$ there is nothing to show.  So let
  $p=\operatorname{char}K>0$.
  We start with a small observation \cite[Exe.~4 p.V.165]{bourbakifields}: 
     If
  $F/E/K$ is a tower of field extensions, with $F/K$ separable, then if $E^pK=E$,
  then $F/E$ is separable.  To prove this, it suffices to show that the natural
  map $F^p\otimes_{E^p} E \to F^pE$ is injective. By the assumption $E=E^pK$, we
  therefore must show $F^p \otimes_{E^p}E^pK\to F^p(E^pK)$ is injective.   Since
  $E^p/K$ is separable~\cite[Prop.~8, p.V.116]{bourbakifields},
  we have that $E^p\otimes_{K^p}K\hookrightarrow E^pK$ is injective.   Since
  field extensions are (faithfully) flat, tensoring by $F^p\otimes_{E^p}(-)$ we
  obtain
  \begin{equation}\label{E:L/KsepCl}
  F^p \otimes_{E^p} (E^p \otimes_{K^p}K) \hookrightarrow F^p\otimes_{E^p}E^pK \to
  F^p(E^pK).
  \end{equation}
  The composition is identified with the map $F^p\otimes_{K^p}K\to F^pK\subseteq
  F^p(E^pK)$, which is injective since $F/K$ is assumed to be separable.  However,
  since $E^pK$ is the field of fractions of $E^p \otimes_{K^p} K$ under the
  inclusion $E^p \otimes_{K^p} K\hookrightarrow E^pK$, we see that the right hand
  map $ F^p \otimes_{E^p} E^pK \to F^p (E^pK)$ in \eqref{E:L/KsepCl} is injective,
  as claimed, since it is obtained from the composition   $F^p \otimes_{E^p} (E^p
  \otimes_{K^p}K)  \to F^p(E^pK)$ in~\eqref{E:L/KsepCl} by localization.

  To prove the lemma, we apply the observation in the previous paragraph with
  $F=L^{\sep}$ and $E=K^{\sep}$.  Thus we just need to show that
  $(K^{\sep})^pK=K^{\sep}$.    Thus we have reduced to the following: if $E/K$
  is a separable algebraic extension,  then $E^pK=E$. Indeed, we have a tower of
  extensions $E/E^pK / K$. The extension $E/E^pK$ is purely inseparable (the
  $p$-th power of every element of $E$ belongs to $E^pK$) while the extension
  $E/K$ is separable. This implies $E=E^pK$.
 \end{proof}

 \subsection{$L/K$-images} 

 Let $L/K$ be a primary extension of fields, \emph{i.e.}, the algebraic
 closure of $K$ in $L$ is purely inseparable over $K$, or equivalently, 
 $K$ equals its separable closure in $L$.  Suppose $A/L$
 is an abelian variety.  The $L/K$-image of $A$ is a pair
 $(\im_{L/K}(A),\lambda)$ consisting of an abelian variety
 $\im_{L/K}(A)$ over $K$ and a homomorphism
 $\lambda : A\to (\im_{L/K}A)_L$ of abelian varieties over $L$ that is
 initial for pairs $(B,f)$ consisting of an abelian variety $B$ over
 $K$ and a homomorphism $f:A\to B_L$:  
 \begin{equation}\label{E:lambda}
 \xymatrix{
 A \ar[rd]_f \ar[r]^<>(0.5)\lambda& (\im_{L/K}A)_L \ar@{-->}[d]^{\exists !}\\
& B_L  
 }
 \end{equation}
 The idea of such an image (and
 the complementary notion of the trace, which is final for 
  pairs $(B,f)$ consisting of an abelian variety $B$ over
 $K$ and a homomorphism $f:B_L\to A$) goes back to Chow, but we appeal to \cite{conradtrace} as a
 modern and comprehensive reference.  The existence of $\im_{L/K}(A)$
 is proven in \cite[Thm.~4.1]{conradtrace}. 
         
  For later reference, given a separable extension $M/K$, and an algebraically disjoint extension $L/K$ (\cite[Def.~5, p.V.108]{bourbakifields}), we have that $LM/L$ is separable \cite[Prop.~5, p.V.131]{bourbakifields}.
  Similarly, given a purely inseparable extension $L/K$ and an arbitrary
    extension $M/K$, we have that $LM/M$ is purely inseparable.   
    In other words, if $M/K$ is separable, and $L/K$ is purely inseparable (and therefore algebraic, so that $L$ is algebraically disjoint from $M$),  then we have a tower:
 \begin{equation}\label{E:tower}
 \xymatrix@R=.5em{
   LM \ar@{-}[dd]_{\text{sep}} \ar@{-}[dr]^{\text{insep}} \\
   & M \ar@{-}[dd]^{\text{sep}}\\
   L\ar@{-}[dr]_{\text{insep}}&\\
   &K
  }
\end{equation}
where ``insep'' means ``purely inseparable''.

       One fact we will use
 later is that formation of the image is insensitive to separable field
        extensions.
          Indeed, a
 special case of \cite[Thm.~5.4]{conradtrace} states:

 \begin{lem}
  \label{L:imageabvar}
  If $L/K$ is purely inseparable, if $M/K$ is separable, and $A/L$ is
  an abelian variety, then
  \[
 \im_{LM/M}(A_{LM})  \iso  (\im_{L/K}(A))_{M}. \qed
  \]
\end{lem}

In fact, we will want a small strengthening of this lemma (\Cref{C:ImAbVarSt}).  To obtain this strengthening, we will first need a few more small results.  First, we will need a slight variation on 
Mumford's Rigidity Lemma \cite[Prop.~6.1(1)]{mumfordGIT}.  If $V$ is a scheme, we use $\abs V$ to denote the underlying topological space.

\begin{lem}[Rigidity Lemma]
  \label{L:GITrigidity}
   Given a diagram
\[
\xymatrix{
X \ar[rr]^f \ar[dr]^p & &Y \ar[dl]_q \\
& S \ar@/^1pc/[ul]^\epsilon}
\]
where $S$ is a Noetherian scheme and:
\begin{enumerate}[label={(\alph*)}]
\item $p_*\mathcal O_X \iso \mathcal O_S$;  \label{L:GITrig-1}
\item $\epsilon$ is a section of $p$, and $\abs S$ consists of a single point, $s$; and 
\item the set-theoretic image $f(\abs{X_s})$ is a single point of $|Y|$; 
 
\end{enumerate}
Then there exists a section $\eta:S \to Y$ of $q$ such that $f = \eta
\circ p$:
\[
\xymatrix{
X \ar[rr]^f \ar[dr]^p & &Y \ar[dl]_q \\
& S \ar@/^1pc/[ul]^\epsilon \ar@/_1pc/@{-->}[ur]_\eta}
\]
\end{lem}

\begin{proof}
This is almost verbatim \cite[Prop.~6.1(1)]{mumfordGIT}.  Indeed, Mumford's hypotheses in \cite[Prop.~6.1(1)]{mumfordGIT} are the same, except that our assumption \ref{L:GITrig-1} is replaced  in \cite[Prop.~6.1(1)]{mumfordGIT} by the assumption that  
$p$ be flat and that $H^0(X_s,\mathcal O_{X_s})\iso \kappa(s)$.
However, these two hypotheses are only used in the proof of \cite[Prop.~6.1(1)]{mumfordGIT} at the top of  \cite[p.116]{mumfordGIT}, where
the reader is invited 
 to verify that these conditions imply that
$p_*\mathcal O_X \iso \mathcal O_S$; but this is our hypothesis \ref{L:GITrig-1}.
\end{proof}

Chow's rigidity theorem for abelian varieties (\emph{e.g.,} \cite[Thm.~3.19]{conradtrace}) implies that a morphism of abelian varieties which is defined after  a purely inseparable extension is already defined over the base field.  Here, we use Mumford's rigidity lemma to prove an analogous statement when the source of the morphism is an arbitrary geometrically integral variety.

\begin{pro}
  \label{L:rigidity}
Let $L/K$ be a purely inseparable extension of fields. Let $(U,u)/K$ be a
 pointed geometrically integral 
separated scheme of finite type, let $A/K$ be an abelian
variety, and suppose that   
$g:(U_L, u_L) \to (A_L, \mathbf 0_{A_L})$ is a pointed $L$-morphism.  Then $g$
descends to $K$.
\end{pro}

\begin{proof}
	   Since $U$ and $A$ are of finite type over $K$, so is $g$.  
    Therefore, 
   there is a subextension $L_0 \subseteq L$,
  finite over $K$, over which $g$ is defined.  Replacing $L$ by $L_0$
  if necessary, we may and do assume $L/K$ is finite and purely
inseparable.  Then $\spec(L\otimes_KL)$ is a Noetherian scheme with a
single point.  This point has residue field $L$; let $s: \spec L
\hookrightarrow \spec(L\otimes_KL)$ be its inclusion.

Let $p_i: \spec(L\otimes_KL) \to \spec L$ be the two projections.  As
usual, since $U_L$ is the base change of a $K$-scheme, there is a
canonical isomorphism $p_1^*(U_L) \iso p_2^*(U_L)$, and we simply call
this object $U_{L\otimes_KL}$ .  We similarly define the pullback of
$u$, $A$ and $\mathbf 0_A$ to $L\otimes_KL$.  
      We want to use fpqc descent to show that $g:U_L\to A_L$ descends to $K$; for this we need to show an equality
of morphisms
\begin{equation}\label{E:RigLemE1}
\xymatrix{
p_1^*g \stackrel?=p_2^*g: U_{L\otimes_KL}\ar[r] & A_{L\otimes_KL}.}
\end{equation}
This equality will follow from the Rigidity Lemma (\Cref{L:GITrigidity}), as we will see.  At the moment, however, we have a diagram
\begin{equation}\label{E:RigLemDiagDesc0}
\xymatrix{
 U_{L\otimes_KL} \ar[rr]^{p_1^*g - p_2^* g} \ar[dr]^{} & &A_{L\otimes_KL}\ar[dl]^{} \\
& \spec L\otimes_KL \ar@/^1.5pc/[ul]^{u_{L\otimes_KL}} & 
}
\end{equation}
If $U/K$ were proper, one could easily check that the hypotheses of the Rigidity Lemma held for the diagram \eqref{E:RigLemDiagDesc0} (see the proof below), and then it would follow quickly from the Rigidity Lemma that equality holds in \eqref{E:RigLemE1} (again, see the proof below).  But, since we are not assuming that $U/K$ is proper, we must do a little work first to get around this issue.

To begin, let $\varpi:X\to \spec K$ be a (Nagata) compactification of $U$ \cite{conradNagata}. Since
$X$ is proper over $K$ and contains the geometrically integral scheme
$U$ as an open dense set,  $X$ is geometrically integral (\emph{e.g.},
\cite[Prop.~5.51(iii)]{GW20})  and so
 (\emph{e.g.}, \cite[\href{https://stacks.math.columbia.edu/tag/0FD2}{Lem.~0FD2}]{stacks-project})
$
  \varpi_*\mathcal O_X  \iso \mathcal O_{\spec K}$. 
   We now base change to $\operatorname{Spec}L$.  
Using \cite[Lem.~2.2]{lutkebohmert93} or \cite[Rem.~2.5]{conradNagata}, there is a
$U_{L}$-admissible blowup 
    \[
\xymatrix{
\widetilde X \ar[rr]^{\varpi'} \ar[dr]_{\widetilde\varpi} & &X_{L}\ar[dl]^{\varpi_L} \\
& \spec L
}
\]
such that $g:U_L\to A_L$ extends to a morphism
$$\widetilde{g}:\widetilde X \to A_{L}.$$  Moreover, using the same argument as before, \emph{i.e.}, that 
$\widetilde X$ is proper over $L$ and contains the geometrically integral scheme
$U_L$ as an dense open subset,  we have that  $\widetilde X$ is geometrically integral  so that 
$
  \tilde \varpi_*\mathcal O_{\widetilde X}  \iso \mathcal O_{\spec L}$.

We now base change to $L\otimes_KL$ and obtain a diagram 
\[
\xymatrix{
\widetilde X_{L\otimes_KL} \ar[rr]^{\varpi'_{L\otimes_K L}} \ar[dr]_{\widetilde\varpi_{L\otimes_KL}} & &X_{L\otimes_KL}\ar[dl]^{\varpi _{L\otimes_KL}} \\
& \spec L\otimes_KL
}
\]
Via cohomology and base change for flat base change, and using that $
  \tilde \varpi_*\mathcal O_{\widetilde X}  \iso \mathcal O_{\spec L}$, we have  that
$(\tilde \varpi_{L\otimes_KL})_* \mathcal O_{\widetilde X_{L\otimes_KL}}  \iso \mathcal O_{\spec
                                                L\otimes_KL}$.

Our goal now is to show that
\begin{equation}\label{E:RigLemE2}
\xymatrix{
 p_1^*\tilde g \stackrel?= p_2^*\tilde g: \widetilde X_{L\otimes_KL}\ar[r] & A_{L\otimes_KL}},
\end{equation}
as this will establish \eqref{E:RigLemE1}, and we will be done.
               For this, we will want to apply the Rigidity Lemma (\Cref{L:GITrigidity}) to the diagram
\begin{equation}\label{E:RigLemDiagDesc}
\xymatrix{
\widetilde X_{L\otimes_KL} \ar[rr]^{p_1^*\tilde g - p_2^*\tilde g} \ar[dr]^{\widetilde\varpi_{L\otimes_KL}} & &A_{L\otimes_KL}\ar[dl]^{} \\
& \spec L\otimes_KL \ar@/^1.5pc/[ul]^{u_{L\otimes_KL}} & 
}
\end{equation}
However, to apply the Rigidity Lemma we still need to check that    $( {p_1^*\tilde g}-{p_2^*\tilde g})(|(\widetilde X_{L\otimes_KL})_s|)$ is set-theoretically a single point of $|(A_{L\otimes_K L})_s|$. 
To see this, we start with the observation that the fiber $(U_{L\otimes_KL})_s =
s^*U_{L\otimes_KL}$ is canonically isomorphic to $U_L$, and
$p_i^*g|_{(U_{L\otimes_KL})_s}  = g$.  Similarly, $|(A_{L\otimes_K L})_s|=|A_L|$.
Our next claim is that $|(U_{L\otimes_KL})_s|$ is dense in $|(\widetilde X_{L\otimes_KL})_s|$, but this just follows since $U_L$ is dense in $\widetilde X$ by construction.  
  Now,  moving forward, we know that $(p_1^*g-p_2^*g)(|(U_{L\otimes_KL})_s|)=|\mathbf 0_{(A_{L\otimes_KL})_s}|$, where here we are denoting by $|\mathbf 0_{(A_{L\otimes_KL})_s}|$ the support of the image of $\mathbf 0_{(A_{L\otimes_KL})_s}:(\operatorname{Spec}L\otimes_K L)_s\to (A_{L\otimes_KL})_s$. 
 This is only an equality on $|(U_{L\otimes_KL})_s|$. 
 However, we know that $s^*({p_1^*\tilde g}-{p_2^*\tilde g})$, as a continuous map $|(\widetilde X_{L\otimes_K L})_s|\to |(A_{L\otimes_KL})_s|$, must take the closure of  $|(U_{L\otimes_K L})_s|$ to the closure of the point $|\mathbf 0_{(A_{L\otimes_KL})_s}|$. 
 But, since $|\mathbf 0_{(A_{L\otimes_KL})_s}|$ is a closed point of $|(A_{L\otimes_KL})_s|$ and $|(\widetilde X_{L\otimes_K L})_s|$ is the closure of $|(U_{L\otimes_K L})_s|$, we see that $s^*( {p_1^*\tilde g}-{p_2^*\tilde g})(|(\widetilde X_{L\otimes_KL} )_s|)=\mathbf 0_{(A_{L\otimes_KL})_s}\in |(A_{L\otimes_KL})_s|$ is a single point.

 Consequently, we can apply the Rigidity Lemma to  diagram \eqref{E:RigLemDiagDesc}, and we find that 
${p_1^*\tilde g}$ and ${p_2^*\tilde g}$ differ by a section
$\eta$ of  $A_{L\otimes_KL}$ over $\operatorname{Spec}L\otimes_KL$.
It remains to show that this section $\eta$ coincides
  with $\mathbf
0_{A_{L\otimes_KL}}$. 
For this it suffices to show that ${p_1^*\tilde g}$ and ${p_2^*\tilde g}$  are equal along a section of $\widetilde X_{L\otimes_KL}$ over $\operatorname{Spec}L\otimes_KL$, and of course,  it therefore suffices to check equality along a section of $U_{L\otimes_KL}\subseteq \widetilde X_{L\otimes_KL}$ over $\operatorname{Spec}L\otimes_KL$; we will use the section $u_{L\otimes_KL}$. 
 Since $g$ takes $u_L$ to $\mathbf 0_{A_L}$, we have that 
$p_1^*g$ and $p_2^*g$ both take the section $p_1^*u_L=p_2^*u_L=u_{L\otimes_KL}$  of $U_{L\otimes_K L}$ to
$\mathbf 0_{A_{L\otimes_KL}}$, and thus $\eta = \mathbf
0_{A_{L\otimes_KL}}$.  Note that here we have used that $u$ is defined over $K$, to identify $p_1^*u_L=p_2^*u_L=u_{L\otimes_KL}$.  
    \end{proof}

To implement this descent result in the setting we want to use it, we need one more result, which states that rational maps that extend to a morphism after base change of field, extend to a morphism over the ground field, as well. 

\begin{lem}\label{L:RatMapLK}
Let $V$ be a reduced scheme of finite type over a field $K$, let $T/K$
be a separated scheme of finite type, let $U\subseteq V$ be a dense
open subset, and let $L/K$ be an arbitrary extension of fields.  Given a morphism $f:U\to T$ over $K$, such that $f_L:U_L\to T_L$ extends to a morphism $V_L\to T_L$, we have that $f:U\to T$ extends to a morphism $V\to T$ over $K$.
\end{lem}
\begin{proof}
To fix some notation, write $\Gamma_g:X\to X \times_K Y$ for the graph of a $K$-morphism $g:X\to Y$, which is a closed 
embedding if $Y$ is separated, and denote by $\bar g(X)$ the scheme-theoretic image of $X$ under $g$.  
Now, considering the graph $\Gamma_f:U\to U\times_KT$ and the inclusion $\iota:U\hookrightarrow V$, 
we wish to show that the first projection $\bar{(\iota \times 1_T)}(\bar {\Gamma}_f(U))\to V$ is an isomorphism, so that the composition $V\stackrel{\iso}{\longrightarrow} \bar{(\iota \times 1_T)}(\bar {\Gamma}_f(U))\to T$ gives an extension of $f:U\to T$.  
We are given that $f_L$ extends to $\tilde f_L: V_L
  \to T_L$.
  Recalling that the scheme-theoretic image is stable under \emph{flat} base change (\emph{e.g.}, \cite[Prop.~V-8, p.217]{EH2000}), we have $\bar {(\iota\times 1)}((\bar {\Gamma}_f(U))_L)=\bar {(\iota\times 1)}( \bar {\Gamma}_{f_L}(U_L))=\bar {\Gamma}_{\tilde f_L}(V_L)$, where the last equality holds since $V_L$ is reduced.
          Finally, since $\bar {\Gamma}_{\tilde f_L}(V_L)\to V_L$ is an isomorphism, we can deduce that $\bar{(\iota \times 1_T)}(\bar \Gamma_f(U))\to V$ is an isomorphism, since isomorphisms satisfy fpqc descent (\emph{e.g.}, \cite[p.583]{GW20}).
\end{proof}

\begin{lem}\label{L:ImAbVarSt}
  Suppose  $L/K$ is a purely inseparable extension of fields, $A/L$ is
  an abelian variety, $(V,v)/K$ is a $K$-pointed (geometrically) connected
  and geometrically reduced scheme of finite type over $K$, and $f:(V_L,v_L)\to (A,\mathbf 0)$ is a pointed $K$-morphism.  
   Then the composition $\lambda \circ f$ of  pointed $L$-morphisms
\begin{equation}\label{E:ImSt1}
\xymatrix{
(V_L,v_L)\ar[r]^f & (A,\mathbf 0) \ar[r]^<>(0.5)\lambda &((\im_{L/K}A)_L,\mathbf 0)
}
\end{equation}
is initial for compositions of pointed $L$-morphisms $(V_L,v_L)\to (A,\mathbf 0 )\to (B_L,\mathbf 0)$, where $B/K$ is an abelian variety over $K$.

 If, moreover, $V$ admits an open cover $\{(U_i,u_i)\}$ by separated  (geometrically) connected
  and geometrically reduced schemes $U_i$ of finite type over $K$, with each irreducible component of the $U_i$ being geometrically integral and admitting a smooth $K$-point $u_i$,
   then 
the composition $\lambda\circ f$, 
descends to a unique pointed $K$-morphism $$\ubar f:(V,v)\longrightarrow (\im_{L/K}A,\mathbf 0).$$  
  \end{lem}

\begin{proof}
The universal property of \eqref{E:ImSt1} follows from the definition of the $L/K$-image.  All that is left is to show the descent. It suffices to show descent on restriction to each of the $U_i$.  Let $U$ be any of the $U_i$.  

 In fact, it suffices to show descent on the normalization $U^\nu$ of
 $U$.  Indeed, since $U$ is geometrically reduced, it is generically
 smooth (smoothness may be verified fpqc locally on the base, and the base change to the algebraic closure is generically smooth), and therefore, there is a dense open subset $U'\subseteq U$ that is normal, so that the normalization $\nu:U^\nu\to U$ is an isomorphism over $U'$.  If we show that  the morphism from $(U^\nu)_L$ descends, then the morphism from $U'_L$ descends.  Then we use \Cref{L:RatMapLK}.
 
   So we can and will assume that $U$ is normal.  We can then focus on one irreducible component at a time, and we can assume that $U$ is integral, and therefore geometrically integral from our assumptions.  Since each irreducible component of $U$ was assumed to have a smooth $K$-point, this gives a $K$-point on each of the irreducible components of the normalizations.  Now use \Cref{L:rigidity}.
 \end{proof}

\begin{pro}\label{C:ImAbVarSt}
 In the situation of \Cref{L:ImAbVarSt}, 
  if in addition $M/K$ is a separable field extension, then the pointed $LM$-morphisms $f_{LM}$ and $\lambda_{LM}$ obtained by base change of \eqref{E:ImSt1} along $LM/L$ factor as 
\begin{equation}\label{E:ImSt2}
\xymatrix{
(V_{LM},v_{LM})\ar[r]^{f_{LM}} & (A_{LM},\mathbf 0) \ar@/^2pc/[rr]^<>(0.5){\lambda_{LM}}  \ar[r]& (( \im_{LM/M} A_{LM})_{LM},\mathbf 0)  \ar[r]&((\im_{L/K}A)_{LM},\mathbf 0).
}
\end{equation}

 If, moreover, $V$ admits an open cover $\{(U_i,u_i)\}$ by separated  (geometrically) connected
  and geometrically reduced schemes $U_i$ of finite type over $K$, with each irreducible component of the $U_i$ being geometrically integral and admitting a smooth $K$-point $u_i$,
  then, 
excluding morphisms with source or target $A_{LM}$,  the morphisms  in \eqref{E:ImSt2} descend uniquely to $M$ to give pointed $M$-morphisms 
\begin{equation}\label{E:ImSt3}
\xymatrix{
(V_{M},v_{M})\ar[r] \ar@/^2pc/[rr]^{\ubar f_M} & ( \im_{LM/M} A_{LM},\mathbf 0)  \ar[r]^\iso&((\im_{L/K}A)_{M},\mathbf 0)& 
}
\end{equation}
where the morphism on the right in \eqref{E:ImSt3} is the isomorphism in \Cref{L:imageabvar}.
  \end{pro}

\begin{proof}
The factorization \eqref{E:ImSt2} follows from the universal property
in the first part of \Cref{L:ImAbVarSt}  applied to the composition 
$
\xymatrix{
(V_{LM},v_{LM})\ar[r]^{f_{LM}} & (A_{LM},\mathbf 0)  \ar[r]& (( \im_{LM/M} A_{LM})_{LM},\mathbf 0) }
$, and the observation that $(\im_{L/K}A)_{LM}= ((\im_{L/K}A)_{L})_{LM}= ((\im_{L/K}A)_{M})_{LM}$ is obtained by pull back of an abelian variety over $M$.  

The descent in \eqref{E:ImSt3} comes from applying \Cref{L:ImAbVarSt}  to \eqref{E:ImSt2}.   
The fact that the composition in~\eqref{E:ImSt3} is identified with $\ubar f_M$ comes from the fact that the pull back of the composition in \eqref{E:ImSt3} to $LM$ is by definition  $\lambda_{LM}\circ f_{LM}$, and, also by definition, we have $\ubar f_{LM} =(\ubar f_L)_M=(\lambda \circ f)_{LM}=  \lambda_{LM}\circ f_{LM}$, so that the uniqueness of the descent shows that $\ubar f_M$ is the composition in \eqref{E:ImSt3}.

That the second morphism in  \eqref{E:ImSt3} is the isomorphism in \Cref{L:imageabvar} follows from the fact that this is the same descended morphism constructed by Conrad   \cite[Thm.~5.4]{conradtrace}.
\end{proof}

 \section{Proof of Theorem \ref{T:main}} \label{S:pfThmA}

We state a more precise version of Theorem \ref{T:main} here:

\begin{teo}[Separable base change]\label{T:mainAbody}
  Let $V$ be a
  geometrically connected and geometrically reduced 
  scheme of finite type over a field $K$.
    Then the Albanese data $(\operatorname{Alb}_{V/K},\ \operatorname{Alb}^1_{V/K},\  a_{V/K})$ for $V$  \eqref{E:AlbDat} is stable under \emph{separable} base change
  of field (\S\ref{S:BC}), and if $V$ admits a $K$-point 
   $v\in V(K)$, then 
    pointed Albanese
 data $(\operatorname{Alb}_{V/K}, \ a_{V/K,v})$ for $(V,v)$ \eqref{E:PAlbDat} is stable
  under \emph{separable} base change of field (\S\ref{S:BC}).
\end{teo}

 For \emph{finite} separable extensions, an easy argument shows:

 \begin{lem}
  \label{L:albBCfinsep}
  Let $V/K$ be a geometrically connected and geometrically reduced  
  scheme of finite type over a field $K$.
  If $L/K$ is finite and separable, then the base change morphisms $\beta^1_{V,L/K}$ \eqref{E:beta1} and
  $\beta_{V,L/K}$ \eqref{E:beta} are isomorphisms.
 \end{lem}

 \begin{proof} 
  By virtue of \Cref{L:betavsbeta1}, it suffices to show that $\beta^1_{V,L/K}$ is an isomorphism, and therefore, 
  by the universal property, it suffices to show that if $A/L$ is any
  abelian variety, $T$ is a torsor under $A$, and $\alpha: V_L \to T$ is a
  morphism, then $\alpha$ factors through $a_L:V_L\to (\alb^1_{V/K})_L$.

  Since $L/K$ is finite and separable, the Weil restriction $\res_{L/K}(A)$ is an
  abelian
  variety  (e.g., \cite[\S 1]{milnearithmeticabvars}) 
      and
    $\res_{L/K}(T)$ is a torsor under $\res_{L/K}(A)$. 
    Since $\operatorname{Hom}_K(V,\res_{L/K}(T))=\operatorname{Hom}_L(V_L,T)$ (\emph{i.e.}, the adjoint property of the Weil restriction, \emph{e.g.}, \cite[p.~191, Lem.~1]{BLR}),
  there is associated to $\alpha$ a $K$-morphism $V \to \res_{L/K}(T)$.  By the universal property of
  $\alb_{V/K}$, this factors over $K$ as
  \[
  \xymatrix{
   V \ar[r] \ar[rd]& \alb^1_{V/K} \ar[d] \\
   &\res_{L/K}(T)
  }
  \]
  Again by the adjoint property of $\res_{L/K}$, this induces a diagram over $L$
  \[
  \xymatrix{
   V_L \ar[r] \ar[rd]_\alpha& (\alb^1_{V/K})_L \ar[d] \\
   &T
  }
  \]
  and so $(\alb^1_{V/K})_L$ is the universal torsor receiving a map from
  $V_L$.   
 \end{proof}

For clarity, recall that a normal Noetherian scheme is irreducible if and only if it is connected.
 In particular, if a scheme over $K$ is smooth, then it is geometrically reduced and geometrically connected if and only if it is geometrically integral. 

 \begin{lem}
  \label{L:albopen}
  Let $V/K$ be a smooth geometrically integral
    scheme 
  over a field $K$.
  If $\iota: U\hookrightarrow V$ is an open immersion, then
  the universal morphism indicated with the dashed arrow in the diagram below:
 $$
 \xymatrix{
 U\ar@{^(->}[d]_\iota \ar[r]& \operatorname{Alb}^1_{U/K} \ar@{-->}[d]^\cong\\
 V\ar[r]& \operatorname{Alb}^1_{V/K}
 }
 $$
  is an isomorphism,  equivariant with respect to a canonical
   isomorphism
   $$\alb_{U/K} \iso \alb_{V/K}.$$ In particular, pre-composing with the inclusion $U\hookrightarrow V$ converts Albanese data (resp.~pointed Albanese data) for $V$ into Albanese (resp.~pointed Albanese data) for $U$.  
 \end{lem}

 \begin{proof} 
 It suffices to show that, if $f:U \to T$ is a morphism
  to a torsor under an abelian variety, then $f$ extends to a morphism
  $\tilde f: V \to T$.  If $T$ is an abelian variety,
   this is a special
  case of \cite[\S 8.4, Cor.~6, p.234]{BLR}.  The general case then follows from this since if  $f_L:U_L\to T_L$ extends to a morphism $V_L\to T_L$ for some field extension $L/K$, then $f:U\to T$ extends to a morphism $\tilde f: V\to T$  (\Cref{L:RatMapLK}), and so one reduces to the previous case by base change to a field $L/K$ over which $T_L$ admits an $L$-point.
              \end{proof}

 The next two lemmas establish that for a variety $V$ that admits a smooth
 alteration, formation of the Albanese torsor commutes with separable base
 change:

 \begin{lem}
  \label{L:albBCautLK}
  Let $V/K$ be a geometrically connected and geometrically reduced 
   scheme  of finite type over a field $K$.  Suppose that there is a
  diagram
  \[
  \xymatrix{
   U \ar@{^(->}[r]^\iota \ar[d]^\pi & X \\
   V
  }
  \]
  of $K$-schemes  with $\pi$ dominant, $\iota$ a dense open
   immersion, and $X$ a smooth proper  scheme over $K$ each connected component of which
       is geometrically integral.  Let $L/K$ be any field
  extension such that $L^{\aut(L/K)} = K$.  Then $\beta_{V,L/K}$ and
  $\beta^1_{V,L/K}$ are isomorphisms.
 \end{lem}

 \begin{proof}
                 Let us write $\sqcup_i \iota_i : \bigsqcup_i U_i \to \bigsqcup_i X_i$ for $\iota : U \hookrightarrow X$, with the $X_i$ being the connected components of $X$.
  Note that, by Lemma \ref{L:albopen} and
  Theorem \ref{T:GrConBC-1}, $\beta^1_{U_i,L/K}$ and $\beta_{U_i,L/K}$ are
  isomorphisms for all $i$.
   By the universal property, each Albanese morphism $a_i : U_i \to \alb^1_{U_i/K}$ induces a diagram
  \[
  \xymatrix@R=1.5em{
   U_i \ar[r]^{a_i \ \ } \ar[d]_{\pi_i} & \alb^1_{U_i/K}\ar[d]^{\delta^1_i}\\
   V \ar[r]^{a \ \ }&  \alb^1_{V/K}
  }
  \]
  and each $\delta^1_i$ is equivariant with respect to the induced
  morphism $\delta_i: \alb_{U_i/K} \to \alb_{V/K}$ of abelian
  varieties.  
  Let $\delta : \prod_i \alb_{U_i/K} \to \alb_{V/K}$ be the homomorphism induced by the $\delta_i$.
  We claim that $\delta$ is surjective and that, for any field extension $L/K$ such that $L^{\aut(L/K)} = K$, we have that 
   $$\ker  \big(\delta_L : \prod_i \alb_{(U_i)_L/L} \to \alb_{V_L/L}\big)$$
    is invariant under $\operatorname{Aut}(L/K)$, so that $\ker \delta_L$ descends to $\ker \delta$.  
      
  The surjectivity of $\delta$ can be seen as follows. Choose a finite field extension $M/L$ such that each $(U_i)_M$ acquires an $M$-point. Since the image of $V_M$ in its Albanese variety $\alb_{V_M/M}$ generates $\alb_{V_M/M}$ and since the disjoint union $\bigsqcup_i(U_i)_M$ dominates $V_M$, we see that $\alb_{V_M/M}$ is generated by certain translates of the images of the induced homomorphisms $\alb_{(U_i)_M/M} \to \alb_{V_M/M}$ and it ensues that $\delta_M$, and hence $\delta$, is surjective (as surjective morphisms satisfy fpqc descent; \emph{e.g.}, \cite[p.584]{GW20}).
  
For the $\operatorname{Aut}(L/K)$-invariance of $\ker \delta_L$, we argue as follows.  Let $\sigma\in \operatorname{Aut}(L/K)$, and for an $L$-scheme $Y$, denote by  $Y^\sigma$ the pull-back of $Y$ along $\sigma:\operatorname{Spec}L\to \operatorname{Spec}L$.  
   We want to show there is a canonical $L$-isomorphism $(\ker \delta_L)^\sigma= \ker \delta_L$. For this, consider the diagram
$$
\xymatrix@C=3em{
V_L\ar[r]^<>(0.5){a_L} \ar@{=}[d]& \operatorname{Alb}_{V_L/L} \ar@{-->}[d]& \prod_i\operatorname{Alb}_{(U_i)_L/L}  \ar@{=}[d] \ar@{->>}[l]_<>(0.5){\delta_L}& \ker \delta_L \ar[l]\\
(V_L)^\sigma\ar[r]^<>(0.5){(a_L)^\sigma} & (\operatorname{Alb}_{V_L/L})^\sigma & \prod_i\left(\operatorname{Alb}_{(U_i)_L/L}\right)^\sigma \ar@{->>}[l]_<>(0.5){(\delta_L)^\sigma}& (\ker \delta_L)^\sigma \ar[l]
}
$$
where the dashed arrow is induced by the universal property of the
Albanese. One concludes from a diagram chase that there is a
scheme-theoretic inclusion
$$
\ker \delta_L\subseteq (\ker \delta_L)^\sigma.
$$
Applying the same argument to $\sigma^{-1}$, we see that $\ker \delta_L\subseteq (\ker \delta_L)^{\sigma^{-1}}$, and then applying $\sigma$ to both sides, we have
$
(\ker \delta_L)^\sigma \subseteq ((\ker \delta_L)^{\sigma^{-1}})^\sigma = \ker \delta_L$, so that $(\ker \delta_L)^\sigma =\ker \delta$, as claimed.

  Now, since we have established that $(\ker \delta)_L=\ker \delta_L$, 
   we have, for any field extension $L/K$ such that $L^{\aut(L/K)} = K$, a commutative diagram:
$$
\xymatrix{
\alb_{V_L/L} \ar[d]_{\beta_{V,L/K}}&\big(\prod_i \alb_{(U_i)_L/L}\big) / \ker \delta_L \ar[l]_<>(0.5)\iso  \ar[d]^{\prod_i \beta_{U_i,L/K}}_\iso\\
(\operatorname{Alb}_{V_K/K})_L& \big(\prod_i \alb_{U_i/K}\big)_L / (\ker\delta)_L \ar[l]_<>(0.5)\iso
}
$$  
    showing that    $\beta_{V,L/K}$ is an isomorphism. By \Cref{L:betavsbeta1}, $\beta^1_{V,L/K}$ is then also an
  isomorphism.
                    \end{proof}

 \begin{lem}
  \label{L:albBCsmoothalteration}
  Suppose $U$, $V$, and $X$ are as in Lemma \ref{L:albBCautLK}.  If $L/K$
  is any separable extension, then $\beta_{V,L/K}$ is an isomorphism.
 \end{lem}

 \begin{proof}
  Let $K^{\sep}$ and $L^{\sep}$ be separable closures of, respectively, $K$ and
  $L$, and consider the diagram of separable (thanks to Lemma \ref{L:L/KsepCl})
  field extensions
  \[
  \xymatrix@R=.5em{
   L^{\sep} \ar@{-}[dd] \ar@{-}[dr] \\
   & L \ar@{--}[dd]\\
   K^{\sep}\ar@{-}[dr]&\\
   &K
  }
  \]
  As we have $(K^{\sep})^{\operatorname{Aut}(K^{\sep}/K)} = K$,
  $(L^{\sep})^{\operatorname{Aut}(L^{\sep}/L)} = L$, and by
  Lemma \ref{L:autLKsepclosed} we also have the identification 
  $(L^{\sep})^{\operatorname{Aut}(L^{\sep}/K^{\sep})} = K^{\sep}$, we can
  apply Lemma \ref{L:albBCautLK} to all three extensions with solid segments  in the diagram above.
  Together with  the
  universal property of the Albanese morphism, we therefore obtain the diagram:
  $$
  \xymatrix{\operatorname{Alb}_{V_{L^\sep}/L^{\sep}} \ar[r]_\cong^{\beta_{V_L, L^\sep/L}}\ar[d]^\cong_{\beta_{{V}, L^\sep/K^\sep}} &
   (\operatorname{Alb}_{V_L/L})_{L^{\sep}} \ar[dd]^{(\beta_{V, L/K})_{L^\sep}} \\
   (\operatorname{Alb}_{V_{K^\sep}/K^\sep})_{L^{\sep}} \ar[d]^{\cong}_{(\beta_{V,K^\sep/K})_{L^\sep}}&
   \\
     ((\operatorname{Alb}_{V/K})_{K^{\sep}})_{L^{\sep}} \ar@{=}[r] & ((\operatorname{Alb}_{V/K})_L)_{L^{\sep}}
   }
   $$ 
   It follows that
  $\beta_{V, L/K}$  becomes an isomorphism after base-change to $L^{\sep}$, and
  hence that it
  is an isomorphism.
 \end{proof}

In the case of a purely inseparable extension $L/K$, it turns out that the $L/K$-image explains the Raynaud Example \ref{E:badbc}:

 \begin{teo}[{Theorem \ref{T:insep}}]
  \label{L:albBCinsep}
  Let $V/K$ be a geometrically connected and geometrically reduced 
  scheme of finite type over a field $K$.  Suppose $L/K$ is a       purely inseparable extension.
Then there is a commutative diagram
      \begin{equation}\label{E:TinsepE2}
  \xymatrix{
\operatorname{Alb}_{V_L/L}  \ar[r]^<>(0.5){\beta_{V,L/K}}   \ar[rd]_<>(0.5){\lambda}& (\operatorname{Alb}_{V/K})_L \ar[d]_{\iso}\\
 & (\im_{L/K}\alb_{V_L/L})_L,
}
\end{equation}  
induced by an isomorphism $  \alb_{V/K} \iso \im_{L/K}\alb_{V_L/L}$, where $\lambda$ is the universal morphism in the definition of the $L/K$-image \eqref{E:lambda}.

 If $V$ admits a $K$-point $v$, then  the composition of pointed $L$-morphisms
  \begin{equation}\label{E:Tinsep2}
\xymatrix{
(V_L,v_L) \ar[r]^<>(0.5){a_{V_L/L}}& (\operatorname{Alb}_{V_L/L},\mathbf 0)  \ar[r]^<>(0.5){\lambda}& ((\im_{L/K}\operatorname{Alb}_{V_L/L})_L,\mathbf 0)
}
\end{equation}
is initial for pointed $L$-morphisms $(V_L,v_L)\to (A_L,\mathbf 0 )$, where $A/K$ is an abelian variety over $K$, 
and  \eqref{E:Tinsep2}
descends to $K$ to give a  pointed $K$-morphism
\begin{equation}\label{E:Tinsep3}
\xymatrix@C=4em{
(V,v) \ar[r]^<>(0.5){\ubar {a_{V_L/L}}}& (\im_{L/K}\operatorname{Alb}_{V_L/L},\mathbf 0)
}
\end{equation}
providing Albanese data for $(V,v)$; \emph{i.e.}, $(\im_{L/K}\operatorname{Alb}_{V_L/L},\  \ubar {a_{V_L/L}})$ is pointed Albanese data for $(V,v)$.  
                            \end{teo}

 \begin{proof}  
 We wish to establish   \eqref{E:TinsepE2}, first.  To this end, let $V = \bigcup_{i=1}^n U_i$ be an affine open cover.
  Since $V$ is geometrically reduced, each $U_i$ admits a point over some finite
  separable extension $M/K$, which can be chosen to be independent of $i$.  By Lemmas \ref{L:imageabvar} and
  \ref{L:albBCfinsep}, it suffices to verify the lemma after base change
  to $M$.  Thus, we may and do assume in particular that $V$ admits a $K$-point, and
  consequently  that the Albanese torsor and the Albanese abelian variety
  coincide.  Moreover, each of the $U_i$ is separated, being affine, and geometrically reduced, being contained in $V$.  Moreover, we can take the $U_i$ to be connected, and then, since we are allowed to take finite separable base changes, we may take the $U_i$ to be geometrically connected, as well.  In other words, 
  we may assume that  $V$ admits an open cover $\{(U_i,u_i)\}$ by separated  (geometrically) connected
  and geometrically reduced schemes $U_i$ of finite type over $K$, with each irreducible component of the $U_i$ being geometrically integral and admitting a smooth $K$-point $u_i$.  Moreover, 
  we have reduced to proving, under these hypotheses, the second assertion of the lemma, namely that  \eqref{E:Tinsep2} descends to  \eqref{E:Tinsep3}, and that this gives pointed Albanese data.  

Let
  $a$ be the composite map \eqref{E:Tinsep2}
  \[
  \xymatrix{
   a: V_L \ar[r]^<>(0.5){a_{V_L/L}} & \alb_{V_L/L} \ar[r]^<>(0.5)\lambda & (\im_{L/K}\alb_{V_L/L})_L.
  }
  \]
  From \Cref{L:ImAbVarSt}, 
  it is initial for pointed maps from $V_L$ to the base change to $L$ of abelian varieties
  defined over $K$ (establishing one of the claims of \Cref{L:albBCinsep}), 
  and  descends to a pointed $K$-
  morphism $\ubar a :V \to \im_{L/K}\alb_{V_L/L}$ over $K$. 
  We claim that this implies that $\ubar a:V\to   \im_{L/K}\alb_{V_L/L}$ is a pointed Albanese.  
   Indeed, given a pointed morphism $V\to A$ to an abelian variety $A$, we obtain a unique morphism making the following diagram commute:
$$
\xymatrix{
V_L\ar[r]^<>(0.5){a} \ar[rd]& (\im_{L/K}\alb_{V_L/L})_L \ar@{-->}[d]\\
& A_L}
$$
Then, from Chow rigidity \cite[Thm.~3.19]{conradtrace}, one has a unique morphism making the following diagram commute:
$$
\xymatrix{
V\ar[r]^<>(0.5){\ubar a } \ar[rd]& \im_{L/K}\alb_{V_L/L}\ar@{-->}[d]\\
& A}
$$  
showing that $\ubar a: V\to \im_{L/K}\alb_{V_L/L}$ is the pointed Albanese.  
 \end{proof}

 Finally, we can prove our main result.

 \begin{proof}[Proof of \Cref{T:mainAbody}]
 	By Lemma~\ref{L:albBCfinsep}, after possibly base-changing along a finite separable field extension,  we may and do assume that the irreducible components of $V$ are geometrically integral.  
   We will identify a finite purely \emph{inseparable} extension $L/K$ such that $V$ admits a smooth alteration relative to $L$, and chase Albanese varieties along the diagram of fields \eqref{E:tower}.

  Let $V_i$ be the irreducible components of $V$ and for each $i$ choose an open affine (and so separated) subset $V_i'\subseteq V_i$.
  Using Nagata compactification \cite{conradNagata}, embed $V'_i \hookrightarrow Y_i$
  into a
  proper geometrically integral variety.  Using \cite[Thm.~4.1]{deJong}, there is a
  diagram
  \[
  \xymatrix{
   U_i \ar[d] \ar@{^{(}->}[r] & X_i \ar[d] \\
   V'_i  \ar@{^{(}->}[r] & Y_i
  }
  \]
  in which the vertical arrows are alterations;  moreover, there is a
  finite, purely inseparable extension $L_i/K$ such that the structural morphism $X_i
  \to \spec K$ factors through $\spec L_i$, and $X_i \to \spec L_i$ is
  smooth \cite[Rem.~4.2]{deJong}. 
 The composition $U_i\hookrightarrow X_i \to \operatorname{Spec}L_i$ together with the map $U_i\to V$ determine a unique morphism $U_i\to V_{L_i}$ over $L_i$,   
 giving the following diagram over $L_i$:
    \[
  \xymatrix{
   U_i \ar[d] \ar@{^{(}->}[r] & X_i  \\
   V_{L_i}
  }
  \] 
  Letting $L/K$ be the (purely inseparable) composite of the $L_i$ (in some algebraically closed field containing the $L_i$), base changing to $L$, and then taking unions, \emph{i.e.},  $U := \bigsqcup_i U_i\times_{L_i}L$ and $X:= \bigsqcup_i X_i\times_{L_i} L$,  
  and we obtain a diagram over~$L$:
  \begin{equation}\label{E:UVLX}
  \xymatrix{
   U \ar[d] \ar@{^{(}->}[r] & X  \\
   V_L
  }
  \end{equation}
satisfying the hypotheses of Lemma
  \ref{L:albBCautLK} over $L$.  Indeed, the only thing to check is that $V_L$ is geometrically reduced over $L$, and that $U\to V_L$ is dominant.  The former holds, as for any extension~$L'/L$ we have $(V_L)\times_L L'= (V\times_KL)\times_L L'= (V_K)\times_KL'$.  For the latter, we started with $\bigsqcup_i U_i\to V$ dominant.  
It follows that the composition $U\to \bigsqcup_i U_i\to V$ is dominant.  Moreover, this morphism factors through $V_L\to V$.  
  From say \cite[Prop.~4.35, p.111]{GW20} one has that $V_L\to V$ is injective by virtue of the fact that  $L/K$ is purely inseparable, and one can conclude that $U\to V_L$ is dominant.

   Now, let $M/K$ be a separable extension of fields, and consider the tower of field extensions in~\eqref{E:tower}.  
  We then compute canonical isomorphisms:
       \begin{equation*}
\xymatrix@C3em@R=1.5em{ 
&V_M \ar[rr]^<>(0.5){a_{V_M/M}}\ar@{=}[d]&& \operatorname{Alb}_{V_M/M} \ar@{=}[d]&\\
&V_M \ar[rr]^<>(0.5){\ubar {a_{V_{LM}/LM}}} \ar@{=}[d]&&\im_{LM/M}(\alb_{V_{LM}/LM}) \ar@{=}[d]& \text{(\Cref{L:albBCinsep})}\\
&V_M \ar[rr]^<>(0.5){\ubar{(a_{V_{L}/L})_{LM}}}\ar@{=}[d]&&\im_{LM/M}((\alb_{V_{L}/L})_{LM})  \ar@{=}[d]& \text{(\Cref{L:albBCsmoothalteration}, \eqref{E:UVLX})}\\
&V_M \ar[rr]^<>(0.5){\left(\ubar{a_{V_{L}/L}}\right)_{M}}\ar@{=}[d]&&(\im_{L/K}\alb_{V_{L}/L})_{M}  \ar@{=}[d]& \text{(\Cref{C:ImAbVarSt})}\\
&V_M \ar[rr]^<>(0.5){\left({a_{V/K}}\right)_{M}}&&( \alb_{V/K})_M   &  \text{(\Cref{L:albBCinsep})}
}
\end{equation*}
completing the proof.
 \end{proof}

We note that under the separated hypothesis, combining \Cref{C:Alb-nEx} with  \cite[Thm.~p.4]{schroeralbanese} one has the following base change result:

\begin{pro}[Schr\"oer] \label{P:Schroeer}
If $V$ is a separated scheme of finite type over a field $K$ that admits an Albanese datum (see  \Cref{C:Alb-nEx}\ref{C:Alb-2}), then the Albanese datum is stable under separable base change of field.  For purely inseparable field extensions, if $\Gamma(V,\mathcal O_V)$ is in addition \emph{geometrically}\footnote{The hypothesis that 
$\Gamma(V,\mathcal O_V)$ be \emph{geometrically} reduced is implicit in \cite[Thm.~p.4]{schroeralbanese} for field extensions that are not separable\.; the necessity of this assumption is made clear by \Cref{E:GRX}.}
 reduced then the base change morphism is a universal homeomorphism.
\end{pro}

\begin{proof}
\Cref{C:Alb-nEx}\ref{C:Alb-2} implies that $V$ satisfies the hypotheses of 
 Schr\"oer  \cite[Thm.~p.4]{schroeralbanese}.
\end{proof}

\section{The universal property of Albanese varieties}
\label{S:univ}

In fact, Theorem \ref{T:GrConBC-1} as stated above is weaker than what
Grothendieck \cite[Thm.~VI.3.3(iii)]{FGA} and Conrad
\cite[Thm.]{ConradMathOver} actually prove:

 \begin{teo}[Grothendieck--Conrad]
  \label{T:groconrad}
  Let $V/K$ be a \emph{proper} geometrically connected and geometrically reduced  scheme
  over a field $K$.  
  Then for any morphism  of schemes $S
  \to \spec K$, and any $S$-morphism  $f: V_S \to P$ to a torsor under an abelian
  scheme $A/S$, there exists a unique $S$-morphism  $g: (\alb^1_{V/K})_S \to
  P$ such that $g\circ a_S = f$.
  If $V$ admits a $K$-point $v$, then  for any  morphism $S
  \to \spec K$, and any pointed $S$-morphism  $f: V_S \to A$ to an abelian
  scheme $A/S$ taking $v_S$ to $\mathbf 0_A$, there exists a unique $S$-homomorphism $g: (\alb_{V/K})_S\to A$ such that $g\circ (a_v)_S = f$.
   \qed
 \end{teo}

 \begin{rem}
 Recall (similarly to \Cref{R:Pic0T}) that if $A$ and $A'$ are abelian schemes over a scheme~$S$, and $P$ and $P'$ are torsors under $A$ and $A'$, respectively, then for any $S$-morphism $g:P\to P'$, there is a unique $S$-homomorphism $\phi:A\to A'$ making $g$ equivariant, and moreover, $g(P)$ is a torsor under $\phi(A)$.  In particular, in the theorem above, 
there is a unique  $S$-homomorphism $(\operatorname{Alb}_{V/K})_S\to A$ making $g:(\operatorname{Alb}^1_{V/K})_S\to P$ equivariant.  
\end{rem}

If one is willing to restrict to base change by smooth morphisms, one
can derive a similar statement without a properness
hypothesis.

 \begin{teo}[Arbitrary separable base change]
  \label{T:groconrad-open}
  Let $V$ be a
  geometrically connected and geometrically reduced 
  scheme of finite type over a field $K$.
      Then for any (inverse limit of) smooth morphism of schemes $S
  \to \spec K$, and any $S$-morphism  $f: V_S \to P$ to a torsor under an abelian
  scheme $A/S$, there exists a unique $S$-homomorphism $(\alb_{V/K})_S\to A$ and a unique equivariant $S$-morphism  $g: (\alb^1_{V/K})_S \to P$ such that $g\circ a_S = f$.
  
  If $V$ admits a $K$-point $v$, then  for any  (inverse limit of) smooth  morphism $S
  \to \spec K$, and any pointed $S$-morphism  $f: V_S \to A$ to an abelian
  scheme $A/S$ taking $v_S$ to~$\mathbf 0_A$, there exists a unique $S$-homomorphism $g: (\alb_{V/K})_S\to A$ such that $g\circ (a_v)_S = f$.
 \end{teo}

 \begin{proof}  We give the proof for the Albanese torsor; the case of the pointed Albanese variety is similar.  It suffices to consider the case where $S$ is irreducible.  By assumption on the morphism $S \to \operatorname{Spec} K$, the extension 
  $\kappa(S)/K$ is separable.
  Consider then the restriction $f_{\eta_S}$ of  $f: X_S \to P$ to the generic
  point $\eta_S$ of $S$. By \Cref{T:mainAbody}, $f_{\eta_S}$ factors
  through $(\mathrm{alb}^1_{V/K})_{\eta_S}$. This gives a canonical
  $\eta_S$-morphism of torsors $(\mathrm{Alb}^1_{V/K})_{\eta_S} \to P_{\eta_S}$
  over $(\mathrm{Alb}_{V/K})_{\eta_S}$.  Let $U\subseteq S$ be an open dense
  subscheme to which these morphisms extend as $g^1: (\alb^1_{V/K})_U \to P_U$
  over $g: (\alb_{V/K})_U \to A_U$.
  By Raynaud's extension theorem \cite[I.2.7]{FC},
  $g$ extends to a morphism of abelian schemes over $S$.
  Let $S' \to S$ be an fpqc morphism such that $(\alb^1_{V/K})_{S'}\to S'$ and
  $T_{S'}\to S'$ admit sections.  Then $(\alb^1_{V/K})_{S'}$ and $P_{S'}$ are
  trivial torsors under abelian schemes over $S'$, and so $g^1_{U\times_S S'}$
  extends to a morphism $(g^1)': (\alb^1_{V/K})_{S'} \to P_{S'}$.  By fpqc descent
  (Lemma~\ref{L:fpqc} below), $(g^1)'$ descends to a morphism $g^1: (\alb^1_{V/K})_S \to
  T$, as desired.
\end{proof}

\begin{lem}
	  \label{L:fpqc}
	  Let $S$ be a scheme  	   and let $X$ and $Y$ be schemes over  $S$, with $Y/S$ separated. 
	  Let $U\subseteq S$ be an open dense subscheme, and let $S' \to S$ be
	  faithfully flat and quasicompact.  Suppose $f: X_U \to Y_U$ is a morphism of
	  schemes over $U$.  If $f_{S'}: X_U \times_S S' \to Y_U\times_S S'$ extends to a
	  morphism $\til f':X_{S'} \to Y_{S'}$, 
	  and $X_U\times_S S'$ is dense in $X_{S'}$,  
     	  then $\til f'$ descends to a morphism
	  $\til f: X \to Y$ over $S$, and $\til f|_U = f$.	
	   	\end{lem}
          
	 \begin{proof}
	  Let $S'' = S' \times_S S'$, equipped with the two projections $p_i:S'' \to
	  S'$. Let $\Gamma_{\til f'} \subseteq X_{S'} \times_{S'} Y_{S'}$ be the graph of
	  $\til f'$ (since the graph morphism $\Gamma_{\til f'}:X_{S'}\to   X_{S'} \times_{S'} Y_{S'}$  is a closed embedding, as $Y/S$ is assumed to be separated, we are identifying the graph morphism with its scheme-theoretic image).  By Grothendieck's theory of fpqc descent (\emph{e.g.}, \cite[\S
	  6.1]{BLR} or \cite[Thm.~3.1]{conradtrace}), it suffices to demonstrate an
	  equality of closed subschemes $p_1^*(\Gamma_{\til f'}) = p_2^*(\Gamma_{\til
	   f'})$. However, $p_i^*(\Gamma_{\til f'})$ contains $p_i^*(\Gamma_{f_{S'}})$ as a
	  dense set (here we are using that $X_U\times_SS'$ is assumed to be dense in $X_{S'}$  and that the scheme-theoretic image is stable under \emph{flat} base change, \emph{e.g.}, \cite[Prop.~V-8, p.217]{EH2000}); and $p_1^*(\Gamma_{f_{S'}}) = p_2^*(\Gamma_{f_{S'}})$, because
	  $f_{S'}$ descends to $f$.
	 \end{proof}

 \bibliographystyle{hamsalpha}
 \bibliography{DCG}

\def\cprime{$'$}
\providecommand{\bysame}{\leavevmode\hbox to3em{\hrulefill}\thinspace}
\providecommand{\MR}{\relax\ifhmode\unskip\space\fi MR }
% \MRhref is called by the amsart/book/proc definition of \MR.
\providecommand{\MRhref}[2]{%
  \href{http://www.ams.org/mathscinet-getitem?mr=#1}{#2}
}
\providecommand{\href}[2]{#2}
\begin{thebibliography}{ACMV23}

\bibitem[AC15]{achtercunningham15}
Jeffrey~D. Achter and Clifton Cunningham, \emph{A note on {$L$}-packets and
  abelian varieties over local fields}, Pacific J. Math. \textbf{273} (2015),
  no.~2, 395--412. \MR{3317772}

\bibitem[ACMV23]{ACMVfunctor}
Jeffrey~D. Achter, Sebastian Casalaina-Martin, and Charles Vial, \emph{A
  functorial approach to regular homomorphisms}, Algebr. Geom. \textbf{10}
  (2023), no.~1, 87--129. \MR{4537125}

\bibitem[AHL13]{aubryhalouilachaud13}
Yves Aubry, Safia Haloui, and Gilles Lachaud, \emph{On the number of points on
  abelian and {J}acobian varieties over finite fields}, Acta Arith.
  \textbf{160} (2013), no.~3, 201--241. \MR{3106095}

\bibitem[BLR90]{BLR}
Siegfried Bosch, Werner L{\"u}tkebohmert, and Michel Raynaud, \emph{N\'eron
  models}, Ergebnisse der Mathematik und ihrer Grenzgebiete (3) [Results in
  Mathematics and Related Areas (3)], vol.~21, Springer-Verlag, Berlin, 1990.
  \MR{1045822 (91i:14034)}

\bibitem[Bou81]{bourbakifields}
Nicolas Bourbaki, \emph{\'{E}l\'{e}ments de math\'{e}matique}, Masson, Paris,
  1981, Alg\`ebre. Chapitres 4 \`a 7. [Algebra. Chapters 4--7]. \MR{643362}

\bibitem[Bri17]{brionstructure}
Michel Brion, \emph{Some structure theorems for algebraic groups}, Algebraic
  groups: structure and actions, Proc. Sympos. Pure Math., vol.~94, Amer. Math.
  Soc., Providence, RI, 2017, pp.~53--126. \MR{3645068}

\bibitem[Che60]{Chevalley60}
C.~Chevalley, \emph{Sur la th\'{e}orie de la vari\'{e}t\'{e} de {P}icard},
  Amer. J. Math. \textbf{82} (1960), 435--490. \MR{118723}

\bibitem[Con06]{conradtrace}
Brian Conrad, \emph{Chow's {$K/k$}-image and {$K/k$}-trace, and the
  {L}ang-{N}\'eron theorem}, Enseign. Math. (2) \textbf{52} (2006), no.~1-2,
  37--108. \MR{2255529 (2007e:14068)}

\bibitem[Con07]{conradNagata}
\bysame, \emph{Deligne's notes on {N}agata compactifications}, J. Ramanujan
  Math. Soc. \textbf{22} (2007), no.~3, 205--257. \MR{2356346}

\bibitem[Con17]{ConradMathOver}
\bysame, \emph{Albanese variety over non-perfect fields}, 2017, Available at
  \url{https://mathoverflow.net/questions/260982/albanese-variety-over-non-perfect-fields},
  p.~1.

\bibitem[DG70]{SGA32}
M.~Demazure and A.~Grothendieck, \emph{Sch\'{e}mas en groupes. {II}: {G}roupes
  de type multiplicatif, et structure des sch\'{e}mas en groupes
  g\'{e}n\'{e}raux}, S\'{e}minaire de G\'{e}om\'{e}trie Alg\'{e}brique du Bois
  Marie 1962/64 (SGA 3). Dirig\'{e} par M. Demazure et A. Grothendieck. Lecture
  Notes in Mathematics, Vol. 152, Springer-Verlag, Berlin-New York, 1970.
  \MR{0274459}

\bibitem[dJ96]{deJong}
A.~J. de~Jong, \emph{Smoothness, semi-stability and alterations}, Inst. Hautes
  \'Etudes Sci. Publ. Math. (1996), no.~83, 51--93. \MR{1423020}

\bibitem[EH00]{EH2000}
David Eisenbud and Joe Harris, \emph{The geometry of schemes}, Graduate Texts
  in Mathematics, vol. 197, Springer-Verlag, New York, 2000. \MR{1730819}

\bibitem[FC90]{FC}
Gerd Faltings and Ching-Li Chai, \emph{Degeneration of abelian varieties},
  Ergebnisse der Mathematik und ihrer Grenzgebiete (3) [Results in Mathematics
  and Related Areas (3)], vol.~22, Springer-Verlag, Berlin, 1990, With an
  appendix by David Mumford. \MR{1083353}

\bibitem[FT15]{flynntesta15}
E.~V. Flynn and D.~Testa, \emph{Finite {W}eil restriction of curves}, Monatsh.
  Math. \textbf{176} (2015), no.~2, 197--218. \MR{3302155}

\bibitem[Gro62]{FGA}
Alexander Grothendieck, \emph{Fondements de la g\'{e}om\'{e}trie
  alg\'{e}brique. [{E}xtraits du {S}\'{e}minaire {B}ourbaki, 1957--1962.]},
  Secr\'{e}tariat math\'{e}matique, Paris, 1962. \MR{0146040}

\bibitem[Gro67]{EGAIV4}
A.~Grothendieck, \emph{\'{E}l\'{e}ments de g\'{e}om\'{e}trie alg\'{e}brique.
  {IV}. \'{E}tude locale des sch\'{e}mas et des morphismes de sch\'{e}mas
  {IV}}, Inst. Hautes \'{E}tudes Sci. Publ. Math. (1967), no.~32, 361.
  \MR{238860}

\bibitem[GW20]{GW20}
Ulrich G\"{o}rtz and Torsten Wedhorn, \emph{Algebraic geometry {I}.
  {S}chemes---with examples and exercises}, Springer Studium
  Mathematik---Master, Springer Spektrum, Wiesbaden, [2020] \copyright 2020,
  Second edition [of 2675155]. \MR{4225278}

\bibitem[Hal08]{hallbigmonodromy}
Chris Hall, \emph{Big symplectic or orthogonal monodromy modulo {$l$}}, Duke
  Math. J. \textbf{141} (2008), no.~1, 179--203. \MR{2372151}

\bibitem[HZ02]{howezhu02}
Everett~W. Howe and Hui~June Zhu, \emph{On the existence of absolutely simple
  abelian varieties of a given dimension over an arbitrary field}, J. Number
  Theory \textbf{92} (2002), no.~1, 139--163. \MR{1880590}

\bibitem[Kat19]{katzbigmonodromy}
Nicholas~M. Katz, \emph{A note on {G}alois representations with big image},
  Enseign. Math. \textbf{65} (2019), no.~3-4, 271--301. \MR{4113044}

\bibitem[Ked22]{kedlayarelclassoneI}
Kiran~S. Kedlaya, \emph{The relative class number one problem for function
  fields, {I}}, Res. Number Theory \textbf{8} (2022), no.~4, Paper No. 79, 21.
  \MR{4493405}

\bibitem[LS21]{laurentschroeralbanese}
Bruno Laurent and Stefan Schr\"oer, \emph{Para-abelian varieties and {A}lbanese
  maps}, \mbox{arXiv:2101.10829 [math.AG]}, 2021.

\bibitem[L{\"u}t93]{lutkebohmert93}
W.~L{\"u}tkebohmert, \emph{On compactification of schemes}, Manuscripta Math.
  \textbf{80} (1993), no.~1, 95--111. \MR{1226600}

\bibitem[MFK94]{mumfordGIT}
D.~Mumford, J.~Fogarty, and F.~Kirwan, \emph{Geometric invariant theory}, third
  ed., Ergebnisse der Mathematik und ihrer Grenzgebiete (2) [Results in
  Mathematics and Related Areas (2)], vol.~34, Springer-Verlag, Berlin, 1994.
  \MR{1304906}

\bibitem[Mil72]{milnearithmeticabvars}
J.~S. Milne, \emph{On the arithmetic of abelian varieties}, Invent. Math.
  \textbf{17} (1972), 177--190. \MR{330174}

\bibitem[ML39]{maclane}
Saunders Mac~Lane, \emph{Modular fields. {I}. {S}eparating transcendence
  bases}, Duke Math. J. \textbf{5} (1939), no.~2, 372--393. \MR{1546131}

\bibitem[Moc12]{mochizuki12}
Shinichi Mochizuki, \emph{Topics in absolute anabelian geometry {I}:
  generalities}, J. Math. Sci. Univ. Tokyo \textbf{19} (2012), no.~2, 139--242.
  \MR{2987306}

\bibitem[NS52]{NS52}
Andr\'{e} N\'{e}ron and Pierre Samuel, \emph{La vari\'{e}t\'{e} de {P}icard
  d'une vari\'{e}t\'{e} normale}, Ann. Inst. Fourier (Grenoble) \textbf{4}
  (1952), 1--30 (1954). \MR{61858}

\bibitem[Ols08]{olssonAV}
Martin~C. Olsson, \emph{Compactifying moduli spaces for abelian varieties},
  Lecture Notes in Mathematics, vol. 1958, Springer-Verlag, Berlin, 2008.
  \MR{2446415}

\bibitem[Sch23]{schroeralbanese}
Stefan Schröer, \emph{{Albanese Maps for Open Algebraic Spaces}},
  International Mathematics Research Notices (2023), rnad187,
  \mbox{https://academic.oup.com/imrn/advance-article-pdf/doi/10.1093/imrn/rnad187/51468824/rnad187.pdf}.

\bibitem[Ser60]{serrealb}
J.-P. Serre, \emph{Morphismes universels et vari\'et\'es d'{A}lbanese},
  S\'{e}minaire {C}. {C}hevalley, 3i\`eme ann\'{e}e: 1958/59, vol.~4, \'{E}cole
  Normale Sup\'{e}rieure, Paris, 1960. \MR{0157865}

\bibitem[{Sta}22]{stacks-project}
The {Stacks Project Authors}, \emph{\itshape {S}tacks {P}roject},
  \url{http://stacks.math.columbia.edu}, 2022.

\bibitem[Voj21]{vojta}
P.~Vojta, \emph{Lecture notes},
  \url{https://math.berkeley.edu/~vojta/254b/ho3.pdf}, 2021.

\bibitem[Wit08]{wittenberg08}
Olivier Wittenberg, \emph{On {A}lbanese torsors and the elementary
  obstruction}, Math. Ann. \textbf{340} (2008), no.~4, 805--838. \MR{2372739}

\end{thebibliography}
  \end{document}